\documentclass[reqno]{amsart}
\usepackage{color}
\usepackage{amsmath,amssymb,amsfonts,amsthm,bm}
\usepackage{mathrsfs}
\usepackage{graphicx}
\usepackage{enumerate}
\usepackage[numbers, sort&compress]{natbib}
\usepackage[shortlabels]{enumitem}
\usepackage[paperwidth=157mm,paperheight=235mm,tmargin=20mm,lmargin=18mm,rmargin=18mm,bmargin=22mm,asymmetric]{geometry}

  \usepackage{newtxtext}
  \usepackage[varvw]{newtxmath}
\newtheorem{theorem}{Theorem}[section]

\newtheorem{proposition}[theorem]{Proposition}
\theoremstyle{assumption}

\theoremstyle{definition}

\theoremstyle{remark}
\newtheorem{remark}[theorem]{Remark}
\numberwithin{equation}{section}

\newcommand{\eps}{\varepsilon}
\newcommand{\norm}[1]{\Vert#1\Vert}
\newcommand{\abs}[1]{\left\vert#1\right\vert}

\newcommand{\inner}[1]{\left(#1\right)}
\newcommand{\comi}[1]{\left<#1\right>}

\newcommand{\normm}[1]{{ \vert\kern-0.25ex \vert\kern-0.25ex \vert #1
		\vert\kern-0.25ex \vert\kern-0.25ex \vert}}

\makeatletter

 \newbox \abstractbox
\renewenvironment{abstract}{\global\setbox\abstractbox=\vbox\bgroup
 \hsize=\textwidth
  \vskip 1.2cm
  \noindent\unskip \textbf{Abstract.}
 }
 {%\vskip12pt \hrule
 \egroup}

\@namedef{subjclassname@2020}{%
	\textup{2020} Mathematics Subject Classification}

 \def\@startsection#1#2#3#4#5#6{%
 \if@noskipsec \leavevmode \fi
 \par \@tempskipa #4\relax
 \@afterindentfalse
 \ifdim \@tempskipa <\z@ \@tempskipa -\@tempskipa \@afterindentfalse\fi
 \if@nobreak \everypar{}\else
     \addpenalty\@secpenalty\addvspace\@tempskipa\fi
 \@ifstar{\@dblarg{\@sect{#1}{\@m}{#3}{#4}{#5}{#6}}}%
         {\@dblarg{\@sect{#1}{#2}{#3}{#4}{#5}{#6}}}%
}

\def\@settitle{%
  \bgroup
  \centering
  \vglue1cm
  \fontsize{12}{15}\fontseries{b}\selectfont
  \@title
  \vskip 20pt plus 6pt minus 8pt
  \egroup
}

\def\@setauthors{%
  \begingroup
  \trivlist
  \centering \bfseries
 \normalsize\@topsep30\p@\relax
  \advance\@topsep by -\baselineskip
  \item\relax
  \andify\authors
 {\rmfamily\authors}%
  \endtrivlist
  \endgroup
}

\def\@setaddresses{\par
  \nobreak \begingroup
\normalsize
  \def\author##1{\nobreak\addvspace\bigskipamount}%
  \def\\{\unskip, \ignorespaces}%
  \interlinepenalty\@M
  \def\address##1##2{\begingroup
    \par\addvspace\bigskipamount\noindent
    \@ifnotempty{##1}{(\ignorespaces##1\unskip) }%
    {\ignorespaces##2}\par\endgroup}%
  \def\curraddr##1##2{\begingroup
    \@ifnotempty{##2}{\nobreak\indent{\itshape Current address}%
      \@ifnotempty{##1}{, \ignorespaces##1\unskip}\/:\space
      ##2\par}\endgroup}%
  \def\email##1##2{\begingroup
    \@ifnotempty{##2}{\nobreak\noindent{\itshape E-mail address}%
      \@ifnotempty{##1}{, \ignorespaces##1\unskip}\/:
       ##2\par}\endgroup}%
   \def\urladdr##1##2{\begingroup
    \@ifnotempty{##2}{\nobreak\indent{\itshape URL}%
      \@ifnotempty{##1}{, \ignorespaces##1\unskip}\/:\space
      \ttfamily##2\par}\endgroup}%
  \addresses
  \endgroup
}

 \renewcommand\section{\@startsection{section}{1}{\z@}%
{27pt plus 6pt minus 8pt}{14pt plus 6pt minus 8pt}%%
{\center\normalfont\large\bfseries}}
\makeatother

\begin{document}
%\begin{frontmatter}

\title[Sharp regularization effect for the  Boltzmann equation]{Sharp regularization effect for the non-cutoff  Boltzmann equation with hard potentials}

\author[J.-L. Chen, W.-X. Li and C.-J. Xu]{ Jun-Ling Chen, Wei-Xi Li \and Chao-Jiang Xu}

\date{}

\address[J.-L. Chen]{
	School of Mathematics and Statistics,
	  Wuhan University,  Wuhan 430072, China
}

\email{
	jun-ling.chen@whu.edu.cn}

\address[W.-X. Li]{School of Mathematics and Statistics,   Wuhan University,  Wuhan 430072, China
	\& Hubei Key Laboratory of Computational Science, Wuhan University, Wuhan 430072,  China
}

\email{
	wei-xi.li@whu.edu.cn}

\address[C.-J. Xu]{School of Mathematics and Key Laboratory of Mathematical MIIT,\\
 Nanjing University of Aeronautics and Astronautics, Nanjing 210016, China
}

\email{xuchaojiang@nuaa.edu.cn}


\begin{abstract} For the Maxwellian molecules or hard potentials case, we verify the smoothing effect for the spatially inhomogeneous Boltzmann equation without angular cutoff. Given initial data with low regularity, we prove solutions  at any positive time are analytic for strong angular singularity,  and in Gevrey class with optimal index for mild angular singularity.  To overcome  the degeneracy in the spatial variable,  a family of well-chosen  vector fields with time-dependent  coefficients  will play a crucial role, and
the sharp regularization  effect of weak solutions  relies on  a quantitative estimate on directional derivatives in  these vector fields.
\end{abstract}

 \maketitle

 %\tableofcontents

\section{Introduction and main result}

Due to the diffusion property, the regularization effect is well explored for parabolic-type equations. As a typical example,   solutions to the Cauchy problem of heat equation will become analytic at positive times for given initial data with low regularity. This kind of parabolic regularization effect has been observed in several classical equations which describe the motion of dilute gas and fluid dynamics in different physical scales.  For instance, in the macroscopic scales, the motion of fluid may be described by the classical Navier-Stokes equations which indeed enjoy the analytic smoothing effect (cf. e.g. Foias-Temam \cite{MR1026858}). Meanwhile, in the mesoscopic kinetic theory, the Boltzmann equation plays a fundamental role, and the regularization properties of  weak solutions were observed in  P.-L. Lions's work \cite{MR1278244} and further verified by L.Desvillettes \cite{MR1324404}. Since then there have been extensive works on the  $C^\infty$-smoothing effect for the non-cutoff Boltzmann equation and related models, most of which are concerned with the spatially homogeneous case;  the breakthrough for the inhomogeneous counterpart was achieved in the very recent work of  Imbert-Silvestre \cite{MR4433077}.
 In this work, we aim to explore the analytic and sharp Gevrey class regularization effect for the spatially inhomogeneous  Boltzmann equation without angular cutoff.  Different from the heat or the    Navier-Stokes equations, the spatially inhomogeneous  Boltzmann equation is a degenerate parabolic equation. Although sometimes we may expect $C^\infty$-smoothness for general degenerate equations,  it is highly non-trivial to get the analytic regularity. In fact for the inhomogeneous Boltzmann equations,  so far very few analytic solutions  are available.

To understand the transport properties of a dilute gas described by the  Boltzmann equation,  explicit solutions would be useful to capture the non-equilibrium phenomena.  Due to the high non-linearity of the Boltzmann collision operator, it is usually not easy to find an explicit solution and in this case, it would be more convenient to solve the Boltzmann equation via the analytic approximation with the help of numerical methods.
In this paper, we will verify theoretically the analyticity at positive time of mild solutions to the spatially inhomogeneous  Boltzmann equation  with strong angular singularity. On the other hand, for mild angular singularity, the sharp regularization that we may expect will be in Gevrey class rather than in analytic space.
    To investigate the sharp regularity, the main difficulty arises from the degeneracy in spatial direction  coupled with the
  highly non-linear feature in the Boltzmann collision operator.  For the spatial homogeneous case, the regularity  issue reduces to a parabolic problem, and motived by the heat equation,  analytic solutions to the Boltzmann equation and related models  have been proven   for rather weak initial data; cf. \cite{MR3177625,MR2557895,MR3665667}  for instance and also \cite{MR3572500, MR3325244, MR3310275, MR2959943, MR2038147, MR3485915,MR1737547, MR2514370} for the regularity in other different kinds of function spaces.  However,  analytic solutions are much less known for the spatially inhomogeneous counterpart, and the  well-posedness  in the analytic space was obtained by S.Ukai  \cite{MR839310} where the author required the analytic regularity  for  initial data so that   Cauchy-Kovalevskaya theorem may apply,   and to the best of our knowledge, no analytic solution is known for non-analytic initial data.  Motived by  the diffusive models such as the hypoelliptic  Fokker-Planck and Landau equations, it is natural to expect a smoothing effect for the spatially inhomogeneous Boltzmann equation in the analytic space or sharp Gevrey class rather than in $C^\infty$ setting.

  The   spatially inhomogeneous  Boltzmann equation in torus reads as
\begin{equation}\label{1}
\begin{array}{ll}
\partial_{t}F+v\cdot\partial_{x}F = Q(F,F), \quad F|_{t=0}=F_0,
\end{array}
\end{equation}
where $F(t, x, v)$ stands for  probability density function  at position $x\in \mathbb{T}^3$, time $t\geq 0$ with  velocity $v\in \mathbb{R}^3$. If $F=F(t,v)$ is independent of $x$, then equation \eqref{1} reduces to the spatial homogeneous Boltzmann equation. 
The Boltzmann collision operator on the right-hand side of \eqref{1} is a bilinear operator defined  by
\begin{equation}\label{collis}
Q(G,F)(t, x, v)=\int_{\mathbb{R}^3}\int_{\mathbb S^2}B(v-v_*,\sigma)(G^{\prime}_{*}F^{\prime}-G_*F)dv_{*}d\sigma,
\end{equation}
where  and throughout the paper we use the
standard shorthand  $F^{\prime}=F(t,x,v^{\prime}), ~F=F(t,x,v), ~G^{\prime}_{*}=G(t,x,v^{\prime}_{*})$ and $G_*=G(t,x,v_*)$, and the pairs  $(v, v_*)$ and   $(v^{\prime}, v^{\prime}_*) $   are  the velocities of  particles after and before collisions, with the following momentum and  energy conservation rules fulfilled:
\begin{equation*}
      v^{\prime}+v^{\prime}_{*}=v+v_{*} , ~~|v^{\prime}|^{2}+|v^{\prime}_{*}|^{2}=|v|^{2}+|v_{*}|^{2}.
\end{equation*}
From the above relations   we have  the so-called $\sigma$-representation,  with $\sigma\in\mathbb S^{2},$
 \begin{equation*}
 \left\{
 \begin{aligned}
   &v'  = \frac{v+v_*}{2} + \frac{ |v-v_*|}{2} \sigma, \\
  &v'_* = \frac{v+v_*}{2} - \frac{ |v-v_*|}{2} \sigma.
 \end{aligned}
 \right.
\end{equation*}
The cross-section $B(v-v_*,\sigma)$ in \eqref{collis}  depends on the relative velocity $\abs{v-v_*}$ and the  deviation angle $\theta$  with
 \begin{equation*}
 \cos \theta =   \frac{  v-v_*} {|v-v_*|}\cdot \sigma.
\end{equation*}
 Without loss of generality, we may assume
that $B(v-v_*,\sigma)$ is supported   on $0\leq \theta\leq \pi/2$ such that $\cos\theta\geq 0$ and also assume that it takes the following specific form:
  \begin{equation}\label{kern}
 { B}(v-v_*, \sigma) =  |v-v_*|^\gamma  b (\cos \theta),
 \end{equation}
 where    $|v-v_*|^\gamma$ is called the kinetic part with   $-3<\gamma\leq 1$, and $b (\cos \theta)$ is called the angular part satisfying that
   \begin{equation}\label{angu}
 0 \leq  \sin \theta   b(\cos \theta)  \approx   \theta^{-1-2s}  \end{equation}
for   $0<s<1$,  where here and throughout the paper  $p\approx q$ means  $C^{-1}q\leq p\leq Cq$ for
some  generic constant $C\geq 1$.  So the angular part $b(\cos\theta)$ has singularity near $0$ in the sense that
 \begin{eqnarray*}
	\int_0^{\pi/2}\sin \theta   b(\cos \theta)\ d\theta=+\infty.
\end{eqnarray*}
In the following discussion, by strong angular singularity we mean that $1/2\leq s<1$,  and mild angular singularity means that $0<s<1/2.$
 Recall $\gamma=0$ is the Maxwellian molecules case and meanwhile the cases of  $-3<\gamma< 0$ and   $0< \gamma $ correspond respectively to the soft   potential and the hard potential.
In this text, we will restrict our attention to  the cases of  Maxwellian molecules and hard potential, i.e., $ \gamma\geq 0.$
%  Moreover the singularity in the angular part will give   we ask  the strong singularity in the angular part (i.e., $1/2 \leq s<1$ in \eqref{angu}) will ensure  the analyticity of weak solution at  any positive time, and meanwhile the mild singularity that  $0< s<\frac 12$ will yield the sharp Gevrey smoothing effect with index $1/2s$ of weak solutions.

 We are concerned with the solution to the Boltzmann equation \eqref{1} around  the normalized global Maxwellian
%\begin{equation*}
      $\mu=\mu(v)=(2\pi)^{-3/2}e^{-|v|^{2}/2}$.
%\end{equation*}
Thus, let $F(t,x,v)=\mu+\sqrt{\mu}f(t,x,v)$ and similarly for the  initial datum $F_0$, then the reformulated unknown $f=f(t,x,v)$ satisfies that
\begin{eqnarray}\label{eqforper}
 \partial_{t}f+v\cdot\nabla_{x}f+\mathcal{L} f
      =\Gamma(f,f),\quad f|_{t=0}=f_0,
\end{eqnarray}
with the   linearized collision operator $\mathcal L$ and the non-linear collision operator $\Gamma(\cdot,\cdot)$ respectively given  by
\begin{equation}\label{linearbol}
 \mathcal L f=-\mu^{-1/2}Q(\mu,\sqrt{\mu} f)-\mu^{-1/2}Q(\sqrt{\mu}f,\mu),
 \end{equation}
 and
 \begin{equation}
 \label{def.nlt}
 	\Gamma(g,h)=\mu^{-1/2}Q(\sqrt{\mu}g,\sqrt{\mu}h).
 \end{equation}
 Initiated by \cite{MR1278244, MR1324404}, so far it is  well understood that the angular singularity will lead to the fractional diffusion in velocity so that it is a natural conjecture that the Boltzmann collision operator without cutoff should behave essentially as the fractional Laplacian:
\begin{equation}\label{beh}
	-Q(g,f)\sim C_g (-\Delta_v)^s+l.o.t.
\end{equation}
where l.o.t.  refers to lower-order terms that are easier
to control.  Note \eqref{beh} is   verified rigorously true by Alexandre-Desvillettes-Villani-Wennberg \cite{MR1765272}  where the velocity should vary in a bounded region.    For the global counterpart of \eqref{beh},
 an accurate characterization by fractional Laplacian $(-\Delta_v)^s$ and fractional Laplacian on sphere $(v\wedge \partial_v)^{2s}$ is given by  \cite{MR3950012} with the help of pseudo-differential calculus.      Moreover,  fractional diffusion in the spatial variable $x$ may be also archived due to the non-trivial interaction between the diffusion in velocity and the transport part. Thus even though the spatially inhomogeneous Boltzmann equation is degenerate in the spatial direction, it admits an intrinsic hypoelliptic structure just like the diffusive variants such as the Fokker-Planck equation or the Landau equation. Inspired by the analytic regularization effect observed by  \cite{MR2523694,MR4612704} for these specific diffusive models, it is natural to ask the same phenomena for the Boltzmann equation with strong angular singularity,   and in  this work, we will confirm it by virtue of a family of well-chosen vector fields. Moreover, for the remaining case of mild angular singularity, we verify the Gevrey smoothing effect with sharp index.

Before stating the main result, we first recall the extensive studies on the regularization properties of weak solutions to the spatially inhomogeneous Boltzmann equation.   The mathematical verification of the regularization phenomena may go back to L.Desvillettes \cite{MR1324404} for a one-dimensional model of the Boltzmann equation. Later on, the intrinsic diffusion structure in velocity was proven by Alexandre-Desvillettes-Villani-Wennberg  \cite{MR1765272}.  Since then substantial developments have been achieved, and  here we only mention the works  \cite{MR2506070, MR2679369, MR2847536, MR2807092,MR2784329, MR4107942}  for the $C^\infty$ or Sobolev regularization effect.  The smoothing effect in  more regular Gevrey class with Gevrey index $1+\frac{1}{2s}$ was proven by \cite{MR3348825, MR4147430, MR4356815, MR4375857}, based on the hypoelliptic structure explored in \cite{MR1949176, MR3950012, MR3456819, MR3102561, MR2763329, MR2467026, MR3193940}.  Another effective tool refers to  De Giorgi-Nash-Moser theory,  with the help of a strong averaging lemma that plays a crucial role in capturing the regularizing effect;   this approach applies recently to study the conditional regularity for the spatially inhomogeneous Boltzmann equation with general initial data (cf. \cite{MR4033752,MR4433077,MR4229202,MR4049224,MR3551261,MR4431674} for instance) and the well-posedness  for the close-to-equilibrium problem with polynomial tails (cf.  \cite{MR4201411,MR4526062,MR4431674}).

 \subsection{Notations and function spaces}\label{notafun}

Given two operators $P_1$ and $P_2$ we denote by $[P_1, P_2]$ the commutator between $P_1$ and $P_2$, that is,
$
	[P_1,  P_2]=P_1P_2-P_2P_1.
$

 We denote  by $\hat f$ or $\mathcal F_x f$ the partial Fourier transform of $f(t,x,v)$ with respect to the spatial variable $x \in \mathbb{T}^3$, that is,
$$\hat{f}(t, m, v)=\mathcal F_x f(t, m, v) =\int_{\mathbb{T}^3}e^{-im\cdot x }f(t, x, v)dx,\quad m \in \mathbb{Z}^3,$$
where here and below we use $m \in \mathbb{Z}^3$ to stand for the Fourier dual variable of $x \in \mathbb{T}^3$. Similarly,  $\mathcal F_{x,v} f$ represents  the full Fourier transform of $f(t,x,v)$ with respect to  $(x,v)$ and we will denote by $(m,\eta)$ the Fourier dual variable of $(x,v).$
For the sake of convenience, we will denote by  $\hat{\Gamma}(\hat{f},\hat{g})$ the partial Fourier transform of  $\Gamma(f,g)$ defined in \eqref{def.nlt},    meaning that
\begin{equation*}
\begin{aligned}
&\hat{\Gamma}(\hat{f},\hat{g})(t,m,v):=\mathcal F_x \inner{\Gamma(f,g)}(t,m,v)\\
&
=\int_{\mathbb{R}^3}\int_{\mathbb{S}^{2}} B(v-v_*,\sigma)  \mu^{1\over 2}(v_*) \left([\hat{f}(v_*')*\hat{g}(v')](m)-[\hat{f}(v_*)*\hat{g}(v)](m)\right) d \sigma d v_*,
\end{aligned}
\end{equation*}
where the convolutions are taken with respect to the Fourier variable $m\in\mathbb Z^3$:
\begin{equation}\label{def:conv}
%\begin{aligned}
%{[\hat{f}(u')*\hat{g}(v')]}(k) :& =  \int_{\mathbb Z^3_\ell} \hat{f}(u',k-\ell)\hat{g}(v',\ell)\,d\Sigma(\ell),\\
{[\hat{f}(u)*\hat{g}(v)]}(m):  =  \int_{\mathbb Z^3_\ell}  \hat{f}(t,m-\ell,u)\hat{g}(t,\ell,v)\,d\Sigma(\ell),
%\end{aligned}
\end{equation}
for any velocities $u,v\in {\mathbb R}^3$. Here and below
  $d\Sigma(m)$ stands for the discrete measure on $\mathbb{Z}^3$, i.e.,
  $$\int_{\mathbb{Z}^3}g(m)d\Sigma(m):=\sum_{m \in \mathbb{Z}^3}g(m)$$ for any summable function $g=g(m)$ on $\mathbb{Z}^3$.  When applying Leibniz's formula, it will be convenient to work with the  trilinear operator $\mathcal T$  defined   by
\begin{equation}\label{matht}
\mathcal T (g,h, \omega)=   \iint B(v-v_*,\sigma) \omega_* \inner{g_*'h'-g_*h} dv_*d\sigma,
\end{equation}
where $B$ is given in \eqref{kern}, and $\omega$ is a function of $v$ variable only.
The bilinear operator $\Gamma$ in \eqref{def.nlt} and the above $\mathcal T$ are linked by
\begin{equation}
\label{trb}
	\Gamma(g,h)=\mathcal T(g,h,\mu^{1/2}).
\end{equation}
 Similarly as above  we denote by $\hat{\mathcal{T}}(\hat{g}, \hat h,  \omega)$ the partial Fourier transform of $ \mathcal T ( g , h ,\omega)$ with respect to $x$, that is,  for any functions $\omega=\omega(v)$ of $v$ variable only,
\begin{multline}\label{Foutri}
\hat{\mathcal{T}}(\hat{g}, \hat h,  \omega)(m,v) =\mathcal F_x\Big( \mathcal T ( g , h ,\omega)\Big)(m,v)\\
=\iint B(v-v_*,\sigma)\omega(v_*)\Big([\hat{g}(v^{\prime}_{*})*
\hat{h}(v^{\prime})](m)-[\hat{g}(v_{*})*
\hat{h}(v)](m)\Big)dv_{*}d\sigma
\end{multline}
where the conclusions are taken with respect to the Fourier variable $m\in\mathbb Z^3$, seeing  definition \eqref{def:conv}.

   Throughout the paper, we will use without confusion  $L^2_v$ to stand for the classical Lebesgue space $L^2$ consisting of functions of specified variable $v$.  
  Similarly for   $L^2_{x,v} $. Denote by $H^p_v$   the classical Sobolev space $H^p$ in $v$ variable, and similarly for $H^p_{x,v}$.
We recall the mixed Lebesgue spaces $L^p_mL^q_TL^r_v$ introduced in \cite{MR4230064}, which is defined by
\begin{align*}
	L^p_mL^q_TL^r_v=\big\{ g=g(t,x,v); ~~ \  \norm{g}_{L^p_mL^q_TL^r_v}<+\infty \big\},
\end{align*}
where
\begin{eqnarray*}
\norm{g}_{L^p_mL^q_TL^r_v}:=\left\{
\begin{aligned}
&  \left(\int_{\mathbb{Z}^3}\left(\int_{0}^{T}\norm{ \hat{g}(t,m,\cdot)}_{L^r_v}^{q}dt\right)^{\frac{p}{q}}d\Sigma(m)\right)^{\frac{1}{p}},\quad q< \infty, \\
& \left(\int_{\mathbb{Z}^3}\left(\sup\limits_{0 < t < T}\norm{\hat{g}(t,m,\cdot)}_{L^r_v}\right)^{p}d\Sigma(m)\right)^{\frac{1}{p}},\quad q=\infty,
\end{aligned}
\right.
\end{eqnarray*}
for $1 \leq p, r < \infty$ and $1 \leq q \leq \infty$.   In particular,  
$$L^p_mL^r_v=\Big\{g=g(x,v);~~\  \norm{g}_{L^p_mL^r_v}:= \Big (\int_{\mathbb{Z}^3}\norm{ \hat{g}(m,\cdot)}_{L^r_v}^{p}d\Sigma(m)\Big )^{\frac{1}{p}}<+\infty\Big\},$$
  and 
   \begin{eqnarray*}
	L_m^1=\Big\{g=g(x);~~ \  \norm{g}_{L^1_m }:=\int_{\mathbb Z^3} |\hat g(m)| d\Sigma(m)<+\infty\Big\}.
\end{eqnarray*}
   Finally we recall the triple-norm   $\normm{\cdot}$   introduced by Alexandre-Morimoto-Ukai-Xu-Yang \cite{MR2863853},   defined as
   \begin{equation}\label{trinorm}
 \begin{aligned}
 	\normm f^2:&= \int_{\mathbb R^3}\int_{\mathbb R^3}\int_{\mathbb{S}^2} B(v-v_*,\sigma) \mu_*
  \inner{f-f'}^2\,d\sigma d vd v_*  \\
  &\quad+\int_{\mathbb R^3}\int_{\mathbb R^3}\int_{\mathbb{S}^2}  B(v-v_*,\sigma) f_*^2 \inner{\sqrt{\mu'}-
  	\sqrt{\mu}}^2\,d\sigma d vd v_*.
 \end{aligned}
   \end{equation}
 Note the triple norm is indeed equivalent to the anisotropic norm $|\cdot|_{N^{\gamma,s}}$  introduced in Gressman-Strain \cite{MR2784329}. Both the two norms can be characterized by  an  explicit norm $\|(a^{1/2})^w f\|_{L^2_v}$ with $(a^{1/2})^w$ standing for the Weyl quantization of symbol $a^{1/2}$ (cf. \cite{MR3950012} for detail).   In this text, we will use the above triple norm to avoid the pseudo-differential calculus.

\subsection{Statement of the main result}
Let $L_m^1L_v^2$ and $L_m^1L_T^\infty L_v^2$ be the spaces defined in the previous part. We first recall     the existence and uniqueness of solutions to \eqref{eqforper}  established by Duan-Liu-Sakamoto-Strain \cite{MR4230064}  in the setting of $L_m^1L_v^2$.  Assume that the cross-section satisfies \eqref{kern} and \eqref{angu} with $ 0\leq \gamma  $ and
 $0<  s <1$.    It is proven by  \cite{MR4230064}  that for given initial datum $f_0\in L_m^1L_v^2$ satisfying that
\begin{eqnarray*}
	\norm{f_0}_{L_m^1L_v^2}\leq \epsilon
\end{eqnarray*}
for some sufficiently small constant  $\epsilon>0$,  the non-linear Boltzmann equation \eqref{eqforper} admits a unique global solution in $L_m^1L_T^\infty L_v^2$ for any $T>0.$ Moreover,  the higher-order regularity of the mild solution $f  $ is obtained   by \cite{MR4356815} which says that     $f$ is in Gevrey class $G^{1+\frac{1}{2s}}(\mathbb{T}_x^3 \times \mathbb{R}_v^3)$ for $t>0$. Recall  $f=f(x, v) \in G^{r}(\mathbb{T}_x^3 \times \mathbb{R}_v^3)$  with index $r>0$
if  $f \in C^\infty(\mathbb{T}_x^3 \times \mathbb{R}_v^3)$ and there exists a  constant $C>0$ such that
\begin{eqnarray*}
	\forall\ \alpha, \beta \in \mathbb{Z}_+^3,\quad   \norm{\partial_x^\alpha\partial_v^\beta f}_{L_{x, v}^2} \leq C^{\abs\alpha+\abs\beta +1}\big[ { (\abs\alpha+\abs\beta)}!\big]^r,
\end{eqnarray*}
Here $r$ is called the Gevrey index.  In particular $G^1(\mathbb{T}_x^3 \times \mathbb{R}_v^3)$ is just the space of analytic functions, and  $G^r (\mathbb{T}_x^3 \times \mathbb{R}_v^3)$ with $0<r<1$ is   the space of ultra-analytic functions. We have an equivalent expression of the Gevrey class $G^r(\mathbb Z_x^3\times\mathbb R_v^3) $ by virtue of the  Fourier multiplier $e^{c(-\Delta_x-\Delta_v)^{\frac{1}{2r}}}$ with $  c>0 $ a constant,   that is,  we say $f\in G^r(\mathbb Z_x^3\times\mathbb R_v^3) $    if  \begin{equation}\label{eqgev}
	e^{c(-\Delta_x-\Delta_v)^{\frac{1}{2r}}} f\in L^2_{x,v}. 
\end{equation}
Here   $ e^{c(-\Delta_x-\Delta_v)^{\frac{1}{2r}}} f$ is defined by 
\begin{align*}
	\mathcal F_{x,v} \Big(e^{c(-\Delta_x-\Delta_v)^{\frac{1}{2r}}} f\Big)(m,\eta)=e^{c(|m|^2+|\eta|^2)^{\frac{1}{2r}}}  \mathcal F_{x,v} f(m,\eta),
\end{align*}
recalling  $\mathcal F_{x,v}$ represents  the full Fourier transform with respect to $(x,v)$ and $(m,\eta)$ are the Fourier dual variable of $(x,v)$.  

 This work aims to prove the sharp Gevrey class smoothing effect, improving the previous Gevrey regularity  index $1+\frac{1}{2s}$ in \cite{MR4356815}.  The main result can be stated as follows.

\begin{theorem}\label{mainresult} Let $G^{r}(\mathbb{T}_x^3 \times \mathbb{R}_v^3)$ be the Gevrey space defined above.
	Assume that the cross-section satisfies \eqref{kern} and \eqref{angu} with $   \gamma \geq 0 $ and
 $0< s <1$.  There exists a sufficiently small constant $\epsilon>0$ such that if
 \begin{equation}
 	\label{smint}
 	\norm{f_0}_{L_m^1L_v^2} \leq \epsilon,
 	 \end{equation}
then the  Boltzmann equation  \eqref{eqforper}  admits a global-in-time solution $f$  satisfying  that $f \in G^{\tau}(\mathbb{T}_x^3 \times \mathbb{R}_v^3)$ for all $t >0$, where
\begin{equation}\label{taudef}
	\tau=\max\Big \{\frac{1}{2s}, 1\Big\}.
\end{equation}
Moreover, for any $T\geq 1$ and any number $\lambda$ satisfying that $\lambda>1+\frac{1}{2s}$, there exists a constant $C>0$ depending on $T$ and $\lambda$,  such that
	\begin{equation}\label{alpha1}
	\forall \ \alpha , \beta \in \mathbb{Z}_{+}^3, \quad 	\sup_{0<t\leq T}t^{(\lambda+1)\abs\alpha+ \lambda \abs\beta} \norm{\partial_x^{\alpha}\partial_{v}^{\beta}f(t)}_{L^2_{x,v}} \leq    C^{|\alpha|+|\beta|+1} [(|\alpha|+|\beta|)!]^\tau.
	\end{equation}
\end{theorem}

\begin{remark}
	As to be seen below, our argument relies on the restriction that $\gamma\geq 0.$  It is interesting to extend the result above to the case of soft potentials, which would require some new ideas.  We hope the method in this text may give insights on the  regularity of the soft potentials case and other related topics for more general spatially inhomogeneous Boltzmann equations.  
\end{remark}

\subsection{Sharpness of the Gevrey index} In view of  \eqref{taudef},  we have analytic regularization effect for the strong angular singularity case (i.e., $ 1/2\leq s<1$).  For the mild angular singularity case of  $0<s<1/2$,  only  Gevrey class regularization with  index $\frac{1}{2s}$  can be expected. In this part, we will confirm the sharpness of the Gevrey index through the  some toy models  of the Boltzmann equation.  To do so, we first  consider  the following  fractional Fokker-Planck equation in $\mathbb T_x^3\times\mathbb R_v^3$: 
\begin{equation}\label{fFP2}
\left\{
\begin{aligned}
	&\partial_tg+v\cdot \partial_xg+ (-\Delta_v)^{s}g=0,  \quad 0<s<1,\\
	& g |_{t=0}=g_0\in L_{x,v}^2,
\end{aligned}
\right.	 
 \end{equation}
 which is a toy model of the Boltzmann equation with Maxwellian molecules (i.e., $\gamma=0$ in \eqref{kern}).  
By  
performing   the full Fourier transform, we could  reformulate \eqref{fFP2} as the following transport equation: 
 \begin{equation*}
\left\{
\begin{aligned}
	& \partial_t\mathcal F_{x,v} g-m\cdot \partial_{\eta }\mathcal F_{x,v} g  +|\eta|^{2s} \mathcal F_{x,v} g  =0, \\
	& \mathcal F_{x,v} g |_{t=0}=\mathcal F_{x,v} g_0, 
\end{aligned}
\right.	 
\end{equation*}
 recalling $(m,\eta)$ are the Fourier dual variable of $(x,v)$.   By solving the above transport equation we get  
  an explicit solution $g$ to \eqref{fFP2}  satisfying  that 
\begin{equation}\label{fourie}
	\big(\mathcal F_{x,v}  g\big)(t,m,\eta)=e^{ -\int_{0}^t|\eta+\rho m|^{2s}d\rho }\, \big(\mathcal F_{x,v} g_0\big)(m,\eta+tm).  
\end{equation}
Moreover,  observe the fact that (cf. \cite[Lemma 3.1]{MR2523694} for instance) 
\begin{equation*}
  - t(\abs\eta^2 + t^2\abs  m^2)^s/c_s\leq 	-\int_{0}^t|\eta+\rho m|^{2s}d\rho  \leq   -c_st(\abs \eta^2 + t^2 \abs m^2)^s, 
\end{equation*}
and thus, for any $t>0,$
\begin{equation}\label{equiv}
  	  -  ( \abs m^2+\abs\eta^2 )^s/c_{s,t}\leq 	-\int_{0}^t|\eta+\rho m|^{2s}d\rho  \leq   -c_{s,t} ( \abs m^2+\abs\eta^2)^s, 
  \end{equation}
where   $c_s>0$ is a small constant depending    only on $s$,  and $c_{s,t}>0$ is  a small constant  depending only on $s$ and $t$.  Then, combining \eqref{fourie} and \eqref{equiv} yields that, for any $t>0,$
\begin{align*}
	&\norm{e^{ c_{s,t} (- \Delta_x- \Delta_v)^s}  g(t)}_{L^2_{x,v}}^2\\
	&=  \int_{\mathbb Z^3\times\mathbb R^3} e^{2c_{s,t}(\abs m^2+\abs \eta^2)^{s}} e^{ -2\int_{0}^t|\eta+\rho m|^{2s}d\rho }\, \big|\big(\mathcal F_{x,v} g_0\big)(m,\eta+tm) \big|^2    d\Sigma(m)d\eta \leq \norm{g_0}_{L^2_{x,v}}^2.
\end{align*}
Then, in view of the equivalent definition \eqref{eqgev} of Gevrey space,  
\begin{equation*}
\forall \ t>0, \quad	g(t,\cdot,\cdot)\in G^{\frac{1}{2s}}(\mathbb T_x^3\times\mathbb R_v^3).
\end{equation*}
Next we will show that the Gevrey index $\frac{1}{2s}$ is sharp.  To do so, 
  let $r$  be any given number satisfying   $0<r <\frac{1}{2s}$, and we choose such an initial datum $g_0$ in \eqref{fFP2}  that
  \begin{equation}\label{blup}
  \forall\ \eps>0,\quad \norm{e^{ \eps (- \Delta_x- \Delta_v)^\frac{1}{2r} } g_0}_{L^2_{x,v}}=+\infty,
  \end{equation}
  which means $g_0\notin G^{r}(\mathbb T_x^3\times\mathbb R_v^3).$ Moreover,  for any constant $c_*>0$,  we can find a constant $R$ depending only on $c_*$ and the constant $c_{s,t}$ in \eqref{equiv}, such that
  \begin{align*}
  	( \abs m^2+\abs\eta^2)^s/c_{s,t}\leq \frac{c_*}{2}( \abs m^2+\abs\eta^2)^{\frac{1}{2r}} +R,
  \end{align*}
 due to the fact that $0<r <\frac{1}{2s}$. Thus, with \eqref{equiv}, it follows that
 \begin{align*}
 	e^{2c_*(\abs m^2+\abs\eta^2)^{\frac{1}{2r}}}  e^{ -2\int_{0}^t|\eta+\rho m|^{2s}d\rho }\geq e^{-2R} 	e^{c_*(\abs m^2+\abs\eta^2)^{\frac{1}{2r}}}.
 \end{align*}  
As a result,  we use \eqref{fourie}   to conclude that, for any given  $t>0,$
\begin{align*}
	&\norm{e^{ c_*(- \Delta_x- \Delta_v)^\frac{1}{2r} } g(t)}_{L^2_{x,v}}^2\\
	&=\int_{\mathbb Z^3\times\mathbb R^3} e^{2c_*(\abs m^2+\abs\eta^2)^{\frac{1}{2r}}}  e^{ -2\int_{0}^t|\eta+\rho m|^{2s}d\rho }\, \big|\big(\mathcal F_{x,v} g_0\big)(m,\eta+tm) \big|^2   d\Sigma(m)d\eta\\
	&\geq e^{-2R} \int_{\mathbb Z^3\times\mathbb R^3} 	 e^{c_*  (\abs m^2+ \abs \eta^2)^{\frac{1}{2r}}}  \big|\big(\mathcal F_{x,v} g_0\big)(m,\eta+tm) \big|^2  d\Sigma(m)d\eta,
\end{align*}
which, with \eqref{blup} and the fact
 that
\begin{align*}
\abs m^2+ \abs \eta^2 \geq \frac{  \abs m^2+ \abs {\eta+t m}^2}{2(t^2+1)},
\end{align*}
  implies, for any given $t>0,$ 
\begin{equation*}
\forall \ c_*>0,  \quad 	\norm{e^{ c_*(- \Delta_x- \Delta_v)^\frac{1}{2r} } g(t)}_{L^2_{x,v}}=+\infty.
\end{equation*}
Thus
  $g(t)\notin G^{r}(\mathbb T_x^3\times\mathbb R_v^3)$ for $t>0$, and  we have proven that
    $\frac{1}{2s}$ is the sharp Gevrey index we may expect when investigating the regularization effect for the toy model \eqref{fFP2} of the Boltzmann equation.

    \smallskip
{\it (i) Mild angular singularity case}. For  $0<s<1/2$, we  get  in Theorem \ref{mainresult}  the regularization effect in the sharp Gevrey class  $\frac{1}{2s}$, coinciding with the index  for the toy model    \eqref{fFP2}.        
 
 \smallskip   
{\it (ii) Strong angular singularity and hard potentials}.  For the Boltzmann equation with strong angular singularity and hard potentials,   more approximate  model than  \eqref{fFP2}   is  
\begin{equation}\label{fFP1}
\left\{
\begin{aligned}
	&\partial_tg+v\cdot \partial_xg+\comi v^{\gamma}(-\Delta_v)^{s}g=0, \\
	&  g |_{t=0}=g_0\in L_{x,v}^2, 
\end{aligned}
\right.	 
\end{equation}
where $\comi v:=\inner{1+v^2}^{1/2}$, and $0<\gamma\leq 1, \ 1/2\leq s<1$.  Note the coefficient $\comi v^\gamma=(1+\abs v^2)^{\gamma/2}$ in \eqref{fFP1} is only (locally) analytic but not ultra-analytic for $0<\gamma\leq 1$,  then heuristically it seems reasonable that the ultra-analyticity could not be achievable and the analyticity  should be  the best regularity setting  we may expect for the toy model \eqref{fFP1},   and so is for the original Boltzmann equation. In Theorem \ref{mainresult},  the analytic smoothing effect is indeed confirmed by observing that $\tau=1$ in $\eqref{taudef}$ for $1/2\leq s<1.$ 
 
 \smallskip   
{\it (iii) Strong angular singularity and Maxwellian molecules}. For $\gamma=0$, we model  the Boltzmann equation     by 
  \eqref{fFP2}.  As shown above,  if $1/2\leq s<1$, then the toy model \eqref{fFP2} will admit the smoothing effect in the  ultra-analytic class $G^{\frac{1}{2s}}(\mathbb T_x^3\times\mathbb R_v^3)$ rather than in the analytic setting.  Naturally, we may   expect a similar ultra-analytic  smoothing effect for the Boltzmann equation when $\gamma=0$ and $1/2\leq s<1$, and this remains unknown at moment.  Here we mention the work of Barbaroux-Hundertmark-Ried-Vugalter \cite{MR3665667},  where they considered the spatially homogeneous  Boltzmann equation (i.e., $F=F(t,v)$ is independent of $x$)    and established the  regularization effect in the  Gevrey class with sharp  index $\frac{1}{2s}$ for the case of Maxwellian molecules.

\subsection{Difficulties and Methodologies}    When exploring the analyticity of the spatially inhomogeneous Boltzmann equation,  the main difficulty arises from the degeneracy in the spatial direction. Compared with elliptic equations that  usually admit analytic regularity,  we may only expect Gevrey regularity for general hypoelliptic equations.  For the specific hypoelliptic Boltzmann equation,  when    performing  the standard   energy,  the key part is the treatment of the commutator between   $\partial_{ v}$  and the transport operator $\partial_t+v\cdot\partial_x$, since the spatial derivative $\partial_x$ will be involved in.  To overcome the degeneracy in spatial direction,  we may apply a global pseudo-differential calculus to derive the intrinsic hypoelliptic structure  induced by the non-trivial interaction between the diffusion part  and the transport part.  This hypoellipticity enables us to conclude the smoothing effect in Gevrey space of index $1+\frac{1}{2s}$;  interested readers may refer to \cite{MR3950012,MR4356815} and the references therein.

Inspired by the regularization effect for the toy model 
\eqref{fFP2}, we would expect similar regularity properties for the  Boltzmann equation.  Recently  in   \cite{MR4612704}, the last two authors and Cao verified the analytic smoothing effect for the Landau equation.  This equation can be regarded as a diffusive model of the Boltzmann equation,  obtained as a grazing limit of the latter.  Note that  the linear Landau collision behaves  as  the differential operator $\Delta_v$,  rather than the fractional Laplacian in the Boltzmann counterpart, so   the treatment of the Landau equation is usually   simpler than that of the   Boltzmann equation.  Although  less technicality is involved in the  Landau collision case  than  the Boltzmann counterpart,   the methods developed for the Landau equation may usually apply to the Boltzmann equation with technical modifications.   However, the situation could be quite different if we investigate the analytic or more general Gevrey class regularity of the two equations.  In fact,  to obtain the Gevrey class regularity,  the key and subtle part  is   to derive quantitative estimates with respect to the orders of derivatives,  which is usually hard for the highly non-linear collision terms.  To explore the analytic smoothing effect of the Landau equation,   the argument therein relies crucially on some differential calculus so that Leibniz's formula may apply when handling the non-linear Landau collision part. However, there will be essential difficulties for the Boltzmann collision term   if we apply a similar  argument as that in the Landau equation with modifications, since the Boltzmann collision behaves as a  pseudo-differential rather than a differential operator so that we have to work with pseudo-differential calculus which prevents us to apply  Leibniz's formula.  Precisely,  the analytic smoothing effect of the Landau equation, obtained in \cite{MR4612704}, relies on the  following second-order  differential operator:
 \begin{eqnarray*}
	M
=  -\int^t_{0}|\partial_{v}+\rho \partial_{x}|^2 d\rho
=-  t\Delta_{v}-t^2 \partial_{x}\cdot  \partial_{v}-\frac{t^3}{3}\Delta_{x}, 
\end{eqnarray*}
which is elliptic  in $x$ and $v$ variables. 
The introduction of $M$ is inspired by the explicit solution to the Fokker-Planck equation (i.e., a specific form of equation \eqref{fFP2} with $s=1$).  We could take advantage of the strong diffusion property  (i.e. the heat diffusion $-\Delta_v$) of the Landau collision part to control the 
   commutator between $M$ and the transport operator $\partial_t+v\cdot\partial_x$ which is
\begin{equation}\label{comMT}
	[M, \ \partial_t+v\cdot\partial_x]=\Delta_v,
\end{equation}
 recalling $[\cdot,\, \cdot]$ stands for the commutator between two operators.
Moreover, 
 the quantitative estimates on  the commutators between $M^k, k\in\mathbb Z_+,$ and the non-linear Landau collision part is hard,  but  achievable with the help of Leibniz type formula (see \cite[Lemma 4.2]{MR4612704}). This enables to perform quantitative estimates on $M^kf$ with $k\in\mathbb Z_+$ and then derive, with the help of the ellipticity of $M$,   the analytic regularization effect of the Landau equation.   Note that we can not apply directly  the above operator $M$ to the Boltzmann equation, since the  Boltzmann collision part behaves as a fractional Laplacian $(-\Delta_v)^s, 0<s<1,$ and the diffusion is too weak to control  the commutator \eqref{comMT} between $M$ and the transport operator.  Inspired by the explicit representation \eqref{fourie}, a natural attempt  is to modify $M$ as follows to save the game:
\begin{eqnarray*}
M_s:	=-\int^t_{0}\big(1+|D_{v}+\rho D_{x}|^{2}\big)^s d\rho ,
\quad D_x=\frac{\partial_x}{i} \textrm{ and }D_v=\frac{\partial_v}{i},
\end{eqnarray*} 
where $M_s$ is  a  Fourier multiplier defined by
\begin{align*}
	\mathcal F_{x,v}(M_s f)(t,m,\eta)=-\int^t_{0} \big(1+|\eta+\rho m|^{2}\big)^s d\rho   	(\mathcal F_{x,v} f)(t,m,\eta).
\end{align*}
 Observe 
 \begin{equation*}
 	 \int^t_{0} m\cdot\partial_\eta \big(1+|\eta+\rho m|^{2}\big)^s d\rho= \int^t_{0} \frac{d}{d\rho} \big(1+|\eta+\rho m|^{2}\big)^s d\rho=\big(1+|\eta+t m|^{2}\big)^s -\big(1+|\eta|^{2}\big)^s, 
 \end{equation*}
 which implies,  
 \begin{align*}
 	[\psi(t,m,\eta),\  \partial_t-m\cdot\partial_\eta]=\big(1+|\eta|^{2}\big)^s,\quad \psi(t,m,\eta):= -\int^t_{0} \big(1+|\eta+\rho  m|^{2}\big)^s d\rho,
 \end{align*}
Thus the commutator 
\begin{equation*}
	[M_s, \partial_t+v\cdot\partial_x]=(1-\Delta_v)^s
\end{equation*}
 could be controlled by the diffusive  part of  the Boltzmann collision.  Moreover, we need to handle  the commutator 
  \begin{equation*}
 	M_s^k\Gamma(f,f)-\Gamma(f, M_s^kf), \quad   k\in\mathbb Z_+,
 \end{equation*}
where $\Gamma$   is the non-linear Boltzmann collision operator defined by \eqref{def.nlt}.  
 It is not hard to control the above commutator by constants $C_k$  depending on $k$.
  However, it is quite difficult and seems not possible to get a quantitative upper-bound   
with respect to  $k\in\mathbb Z_+$, saying   
\begin{equation*}
	C^{k+1}(k!)^r 
\end{equation*}
with $C$ a constant independent of $k,$  
 since $M_s$ is a pseudo-differential  rather than a differential operator so that Leibniz's formula can not apply.   
Thus, to handle the non-linear Boltzmann collision part,  it seems reasonable to work with    differential rather than pseudo-differential operators,  so that we could take advantage of Leibniz's formula  as well as induction argument to derive quantitative estimates with respect to derivatives. On the other hand, the classical one-order differential operator $\partial_v$ or $\partial_x$ is not a good choice,  since the Boltzmann equation is degenerate in the spatial variable $x$ and the spatial derivative $\partial_x$  will appear in the commutator between  $ \partial_{ v}$  and the transport operator.  

 The new idea in this text is that   instead of  the sole  $ \partial_{ x}$ or $ \partial_{ v}$,  we  work with the following combination of $\partial_x$ and $\partial_v$ with time-dependent coefficients: 
\begin{equation*}
	\xi(t)\partial_{x_j}+\zeta(t)\partial_{v_j}, \quad 1\leq j\leq 3, 
\end{equation*}  
such that, denoting by $\rho'(t)$ the time derivative of  function $\rho(t),$  
\begin{equation}\label{lascommut}
	[\xi(t)\partial_{x_j}+\zeta(t)\partial_{v_j}, \ \partial_t+v\cdot\partial_x] =-\xi'(t)\partial_{x_j}-\zeta'(t)\partial_{v_j}+\zeta(t)\partial_{x_j}=-\zeta'(t)\partial_{v_j}.
	\end{equation}
	As to be seen in the last two sections, the commutator above indeed  can be controlled by the diffusive Boltzmann operator.   
The choice of $\xi(t)$ and $\zeta(t)$ is flexible, 
provided  $\xi'(t)=\zeta(t)$.
For the sake of simplicity,  we choose  $\xi=(1+\delta)^{-1}t^{\delta+1}$  and $\zeta=t^\delta$  and consider  a family of one-order differential operators $H_\delta $  defined by
   \begin{equation}
   	\label{vecM}
   H_\delta= \frac{1}{\delta+1}t^{\delta+1} \partial_{x_1}+ t^{\delta} \partial_{v_1},
   \end{equation}
   where $\delta$ satisfies that
\begin{equation}\label{relation}
 1+\frac{1}{2s}<\delta.
 \end{equation}
In view of \eqref{lascommut},     the spatial derivatives are not involved in the commutator  between $H_\delta$ and the transport operator, that is,  
  \begin{equation}\label{keyob}
  	[H_\delta, \,\,  \partial_t+v\,\cdot\,\partial_x ]=-\delta t^{\delta-1}\partial_{v_1}.
  \end{equation}
  More generally, we have
 \begin{equation}\label{kehigher}
\forall\ k\geq 1,\quad 	[H_\delta ^k, \,\,  \partial_t+v\,\cdot\,\partial_x ]=-\delta kt^{\delta-1}\partial_{v_1} H_\delta^{k-1},
\end{equation}
which can be derived by using induction on $k$.  In fact, the validity of \eqref{kehigher} for $k=1$ follows from \eqref{keyob}. Now supposing  that
 \begin{equation}\label{elli}
\forall\ \ell\leq k-1, \quad 	[H_\delta^\ell, \,\,  \partial_t+v\,\cdot\,\partial_x ]=-\delta \ell t^{\delta-1}\partial_{v_1} H_\delta^{\ell-1},
\end{equation}
we will prove the validity of \eqref{elli} for $\ell=k\geq 2$. To do so,
 we use \eqref{keyob} and \eqref{kehigher} as well as the fact that
 \begin{align*}
 	[T_1T_2,\  T_3] =T_1T_2T_3-T_3T_1T_2 =T_1[T_2,\ T_3]+[T_1,\ T_3]T_2,
 \end{align*}
to compute
\begin{align*}
	&[H_\delta^k, \,\,  \partial_t+v\,\cdot\,\partial_x ] =[H_\delta H_\delta^{k-1}, \,\,  \partial_t+v\,\cdot\,\partial_x ]\\
	&=H_\delta[H_\delta^{k-1}, \,\,  \partial_t+v\,\cdot\,\partial_x ]+[H_\delta, \,\,  \partial_t+v\,\cdot\,\partial_x ]H_\delta^{k-1}\\
	&=H_\delta\inner{-\delta (k-1) t^{\delta-1}\partial_{v_1} H_\delta^{k-2}}-\delta t^{\delta-1}\partial_{v_1} H_\delta^{k-1}\\
	&=-\delta (k-1) t^{\delta-1}\partial_{v_1} H_\delta^{k-1} -\delta t^{\delta-1}\partial_{v_1} H_\delta^{k-1}=-\delta k t^{\delta-1}\partial_{v_1} H_\delta^{k-1}.
\end{align*}
This gives the validity of \eqref{elli} for $\ell=k$. Thus \eqref{kehigher} holds true for all $k\geq 1$.  This enables us to apply the diffusion in the  velocity direction to obtain a crucial estimate of the directional derivatives $H_\delta^kf$ for solution $f$.
Moreover, the classical derivatives can be generated by the linear combination of $H_\delta$ for suitable  $\delta$ with time-dependent coefficients, so that the desired quantitative estimate on the classical derivatives is  available (see \eqref{generate} below for the explicit formulation).

In this text let   $\lambda$ be an arbitrarily   given number satisfying  \eqref{relation},  that is, $\lambda>1+\frac{1}{2s}$.   We define   $\delta_1$ and $\delta_2$ in terms of $\lambda$ by setting
\begin{equation}\label{de1de2}
	\delta_1 =\lambda, \quad 	\delta_2=\left\{
	\begin{aligned}
		& 1+2s+(1-2s)\lambda,\quad \textrm{ if } \  0< s<1/2.\\
		 &\frac{1}{2}\Big(\lambda+1+\frac{1}{2s}\Big),\quad\textrm{ if } \ 1/2\leq s<1.
	\end{aligned}
	\right.
\end{equation}
By virtue of the fact that $\lambda>1+\frac{1}{2s} $,  direct computation yields that
\begin{equation}\label{lodelta}
	\delta_1>\delta_2>1+\frac{1}{2s},
\end{equation}
So  both $\delta_1$ and $\delta_2$ satisfy \eqref{relation}.  With $\delta_1$ and  $\delta_2$ given above,
  let $H_{\delta_1}$ and  $H_{\delta_2}$ be  defined by \eqref{vecM}:
\begin{equation*}
	H_{\delta_1}=\frac{1}{\delta_1+1} t^{\delta_1+1}\partial_{x_1}+t^{\delta_1}\partial_{v_1},\quad
  H_{\delta_2}=\frac{1}{\delta_2+ 1} t^{\delta_2+1}\partial_{x_1}+t^{\delta_2}\partial_{v_1}.
\end{equation*}
Then  $\partial_{x_1}$ and $\partial_{v_1}$ can be generated by the linear combination of $H_{\delta_j}, j=1,2$, that is,
\begin{equation}
	\label{generate}
\left\{
\begin{aligned}
&t^{\lambda+ 1}\partial_{x_1}=	t^{\delta_1+ 1}\partial_{x_1}=\frac{(\delta_2+ 1)(\delta_1+1)}{\delta_2-\delta_1}  H_{\delta_1}-\frac{(\delta_2+ 1)(\delta_1+1)}{\delta_2-\delta_1} t^{\delta_1-\delta_2}H_{\delta_2},\\
&t^{\lambda}\partial_{v_1}=t^{\delta_1}\partial_{v_1}=-\frac{\delta_1+ 1}{\delta_2-\delta_1} H_{\delta_1}+\frac{\delta_2+ 1}{\delta_2-\delta_1}t^{\delta_1-\delta_2}H_{\delta_2} .
\end{aligned}
\right.
\end{equation}
This enables to control the classical derivatives in terms of the directional derivatives in $H_{\delta_1} $ and $H_{\delta_2} $.

   \subsection{Arrangement of the paper}
    The rest of this paper is arranged as follows. In Section \ref{sec:prelim},  we   recall a few  preliminary  estimates  that will be used throughout the argument.  Section \ref{sec:non-linear}  is devoted to estimating  the commutator between  directional derivatives and collision operator.   The proof of the main result is presented in Sections \ref{sec: strong} and \ref{sec:mild},  where   we  treat,  respectively,  the  strong  angular singularity case  and  the mild one.

 \section{Preliminaries}\label{sec:prelim}

  In this part, we will recall some   estimates to be used later.  Let $\mathcal L$ be the linearized Boltzmann operator in \eqref{linearbol} and let $\normm{\cdot}$ be the triple norm defined by \eqref{trinorm}. Then by the coercive estimate and   identification of the triple norm   (cf.\cite[Propositions 2.1 and 2.2]{MR2863853} for instance), it follows that
 \begin{equation}\label{rela}
\forall\ f\in \mathcal S(\mathbb R_v^3),\quad  c_0\normm f^2 \leq   \inner{\mathcal L f, f}_{L^2 _v} + \norm{f}_{L^2 _v}^2,
  \end{equation}
  and that, for Maxwellian molecules  and hard potential cases that $\gamma\geq 0$,
 \begin{equation}
	\label{+lowoftri}
\forall\ f\in \mathcal S(\mathbb R_v^3),\quad 	c_0 \norm{f}_{H^s_{v}} \leq \normm{f},
\end{equation}
where $s$ is the number in \eqref{angu}, and $c_0>0$ is a small constant, and  $\mathcal S(\mathbb R_v^3)$ stands for the Schwartz Space  in $\mathbb R_v^3$. Note the above estimates still hold true for any $f$ such that $\normm f<+\infty.$  

 For simplicity of notations,   we will use $C_0$, in the following argument,  to denote a generic constant which may vary from line to line by enlarging $C_0$ if necessary.    
Now, we recall the   trilinear estimate of the collision operator,   which says (cf.  \cite[theorem 2.1]{MR2784329}) that, for any $f,g,h\in\mathcal S (\mathbb R_v^3),$
\begin{equation}\label{trin}
\big|\big(  \mathcal T ( f, g, \mu^{1/2}),  \  h\big)_{L^2_v}\big|=\big|\big(  \Gamma ( f, g),  \  h\big)_{L^2_v}\big|\leq C_0\norm{f}_{L_v^2}\cdot \normm{ g} \cdot \normm{ h},\end{equation}
recalling $\mathcal T$ is defined in \eqref{matht}.
Furthermore, we mainly employ the counterpart of the above estimate after performing the partial Fourier transform in $x$ variable.    Then, by \cite[Lemma 3.2]{MR4230064}),  the following estimate
 \begin{multline}\label{upptrifour}
  \big|\big( \hat{\mathcal T} ( \hat f, \hat g, \mu^{1/2}),  \  \hat h\big)_{L^2_v}\big| =	\big|\big( \hat \Gamma{(\hat f(m), \hat g(m))}, \hat h(m)\big)_{L^2_v}\big |\\
  \leq C_0 \normm{\hat h(m)} \int_{\mathbb Z^3 } \norm{\hat f(m-\ell)}_{L^2_v}   \normm{\hat g(\ell)}    d\Sigma(\ell)
  \end{multline}
holds true for any $m\in\mathbb Z^3$ and for   any  $f,g,h\in L_m^1(\mathcal S(\mathbb R_v^3))$.  More generally,  if   $\omega=\omega(v)$  is a given function of $v$ variable  satisfying the condition that there exists a constant $\tilde C>0$ such that
\begin{equation}\label{contild}
  \forall\ v\in\mathbb R^3, \quad |  \omega (v) | \leq
\tilde C  \mu(v)^{1/4},
\end{equation}
 then following the same argument for proving \eqref{trin}, with $\mu^{1/2}$ therein replaced by $\omega,$ gives that
\begin{equation*}
	 \forall\, f,g,h\in\mathcal S (\mathbb R_v^3), \quad
\big|\big(  \mathcal T ( f, g,   \omega),  \  h\big)_{L^2_v}\big|\leq C_0 \tilde C \norm{f}_{L_v^2}\cdot \normm{ g} \cdot \normm{ h}.
\end{equation*}
As a result, similar to \eqref{upptrifour}, we   perform  the partial Fourier transform in $x$ variable to conclude
\begin{equation}\label{ketres}
 		\big|\big( \hat{\mathcal{T}}(\hat{f}, \hat g,  \omega), \hat h(m)\big)_{L^2_v}\big |\leq C_0 \tilde C \normm{\hat h(m)} \int_{\mathbb Z^3 } \norm{\hat f(m-\ell)}_{L^2_v}   \normm{\hat g(\ell)}    d\Sigma(\ell)
 \end{equation}
 with $\tilde C$   the constant in \eqref{contild}.   In particular, if    $g$ in \eqref{ketres} is a function of only  $v$ variable, then \eqref{ketres} reduces to 
\begin{equation*}
		\big|\big( \hat{\mathcal{T}}(\hat{f},   g,  \omega), \hat h(m)\big)_{L^2_v}\big |\leq C_0 \tilde C \norm{\hat f(m)}_{L^2_v}   \normm{  g } \times   \normm{\hat h(m)}. 	\end{equation*}
This, with the fact  that   (cf. \cite[Proposition 2.2]{MR2863853})  
\begin{equation*}
	\normm{ g} \leq \tilde c \norm{ (1+|v|^{2+\gamma}-\Delta_v)  g}_{L_v^2 }
\end{equation*} 
for some constant $\tilde c>0$, yields that, enlarging $C_0$ if necessary, 
\begin{equation*}
	\big|\big( \hat{\mathcal{T}}(\hat{f},   g,  \omega), \hat h(m)\big)_{L^2_v}\big |\leq C_0 \tilde C \norm{\hat f(m)}_{L^2_v}  \norm{ (1+|v|^{2+\gamma}-\Delta_v)  g}_{L_v^2 }\normm{\hat h(m)}. 
\end{equation*} 
 As a result, if   $g=g(v)\in \mathcal S(\mathcal R_v^3)$ is any  function of only $v$ variable,   satisfying  the condition  that  there exists a constant $\tilde C_\gamma>0$, depending only on the number $\gamma$ in \eqref{kern},  such that 
\begin{equation}\label{contild+++}
  \forall\ v\in\mathbb R^3, \ \forall \ k\in\mathbb Z_+, \quad |  (1+|v|^{2+\gamma}-\Delta_v)\partial_v^k g(v) | \leq \tilde C_\gamma L_k
   \mu(v)^{1/8},
\end{equation}  
with $L_k$   constants depending only on $k,$
then  
\begin{equation}\label{trisole}
	\forall\ k\in\mathbb Z_+,\quad 	\big|\big( \hat{\mathcal{T}}(\hat{f},  \partial_v^k g,  \omega), \hat h(m)\big)_{L^2_v}\big |\leq C_0 \tilde C  L_k \norm{\hat f(m)}_{L^2_v}     \normm{\hat h(m)} \end{equation}
by enlarging $C_0$ if necessary,   recalling  $\tilde C$ is the constant  given in \eqref{contild}.   Similarly,  for any functions $ \omega=\omega(v)$ and $g=g(v) $ of only $v$ variable, satisfying \eqref{contild} and \eqref{contild+++}, respectively, we have  
 \begin{equation}\label{tretmate}
	\forall\ k\in\mathbb Z_+,\quad 	\big|\big( \hat{\mathcal{T}}(\partial_v^k g,  \hat{f},    \omega), \hat h(m)\big)_{L^2_v}\big |\leq C_0 \tilde C  L_k \normm{\hat f(m)}\times     \normm{\hat h(m)}. \end{equation}
Finally, we recall an estimate (cf. \cite[Lemma 2.5]{MR4356815})  that will be frequently used to control the non-linear term $\Gamma(f,g)$.
 For an arbitrarily given integer $j_0 \geq 1$, it holds that
	\begin{equation}\label{MF}
	\begin{aligned}
	&\int_{\mathbb{Z}^3}\bigg[\int_{0}^{T}\Big(\int_{\mathbb{Z}^3}\sum_{1 \leq j \leq j_0}\norm{\hat{f}_j(t,m-\ell)}_{L^2_v}\normm{  \hat{g}_j(t,\ell)}d\Sigma(\ell)\Big)^{2}dt\bigg ]^{1/2}d\Sigma(m)\\
	& \leq  \sum_{j=1}^{j_0}\Big(\int_{\mathbb{Z}^3}\sup\limits_{0 < t \leq T}\norm{\hat{f}_j(t,m)}_{L^2_v}d\Sigma(m)\Big)\int_{\mathbb{Z}^3}\Big(\int_{0}^{T}\normm{  \hat{g}_j(t,m)}^{2}dt\Big)^{1\over2}d\Sigma(m),
	\end{aligned}
	\end{equation}
	for any $f_j \in L_m^1L_T^{\infty}L_v^2$ and any $g_j$ such that $\normm{ g_j} \in L_m^1L_T^{2}$ with $1 \leq j \leq j_0$. It can be derived directly by Minkowski's inequality and Fubini's theorem, cf.  \cite[Lemma 2.5]{MR4356815} for detail.

 \section{Commutator estimates}\label{sec:non-linear}
 This part  is devoted to  dealing with the commutator between the directional derivative  $ H_\delta^{k}$ and  the collision part $\Gamma(g,h),$
 recalling $H_\delta$ is defined by \eqref{vecM}.  With the notations in Subsection \ref{notafun},   the results  on commutator estimates can be stated as follows.

\begin{proposition} \label{lemgamma}
Assume that the cross-section satisfies \eqref{kern} and \eqref{angu} with $  \gamma \geq 0 $ and $ 0<s <1$.  Recall $H_{\delta}$   is defined by \eqref{vecM}, with  $\delta$   an arbitrarily given number satisfying \eqref{relation}.   Let  $k\geq 1$ and $T\geq 1$ be  given, and
 	 let $f\in L_m^1L_T^\infty L_v^2$  be any solution to the Cauchy problem \eqref{eqforper} satisfying that
 \begin{equation}\label{apbound}
   \int_{\mathbb{Z}^3}\sup\limits_{0 < t \leq T}\norm{\widehat{H_{\delta}^{k}f}(t,m)}_{L^2_v}d\Sigma(m)
  +      \int_{\mathbb{Z}^3}\Big (\int_{0}^{T}\normm{ \widehat{H_{\delta}^{k}f}(t,m)}^{2}dt\Big )^{1\over2}d\Sigma(m)
<+\infty.
  %\norm{f_0}_{ L^1_kL^2_v}.
\end{equation}
  Suppose that    for any  $j\leq k-1$  we have
\begin{multline}\label{suppose}
\int_{\mathbb{Z}^3}\sup\limits_{0 < t \leq T}\norm{\widehat{H_{\delta}^{j}f}(t,m)}_{L^2_v}d\Sigma(m)\\+ \int_{\mathbb{Z}^3}\left(\int_{0}^{T}\normm{  \widehat{H_{\delta}^{j}f}(t,m)}^{2}dt\right)^{1/2}d\Sigma(m) \leq \frac{\eps_0C_{*}^{j}( j!)^\tau}{(j+1)^2},
\end{multline}
where $\tau\geq 1$ is given in \eqref{taudef}, and $\eps_0,C_*>0$ are two constants.  If    $C_*\geq 4T^\delta $, then   there exists a constant $C$,  depending only on the number  $ C_0$ in   \eqref{ketres} but independent of $k$,  such that  for any $\eps >0$ we have
\begin{equation*}
\begin{aligned}
&\int_{\mathbb{Z}^3}\left(\int_{0}^{T}\big|\big(\mathcal{F}_x(H_{\delta}^{k}\Gamma(f,f)), \widehat{H_{\delta}^{k}f}\big)_{L^2_v}\big|dt\right)^{1/2}d\Sigma(m)
\\ & \leq    C{\eps}^{-1}\eps_0   \int_{\mathbb{Z}^3}\sup\limits_{0 < t \leq T}\norm{\widehat{H_{\delta}^{k}f}(t,m)}_{L^2_v}d\Sigma(m)\\
&\quad +    \inner{\eps+C{\eps}^{-1}\eps_0}  \int_{\mathbb{Z}^3}\Big (\int_{0}^{T}\normm{ \widehat{H_{\delta}^{k}f}(t,m)}^{2}dt\Big )^{1/2}d\Sigma(m) +C \eps^{-1}  \eps_0^2 \frac{C_{*}^{k} (k!)^\tau}{(k+1)^2}.
\end{aligned}
\end{equation*}
\end{proposition}

\begin{remark}
We impose  assumption  \eqref{apbound} to ensure rigorous  rather than formal computations in the proof of Proposition \ref{lemgamma}
	   when performing estimates involving the term $H_\delta^k f.$ 
\end{remark}

\begin{proof}[Proof of Proposition \ref{lemgamma}]  If  no confusion occurs, we will write  in the proof  $H=H_\delta$ for short, omitting the subscript $\delta.$   To simplify the notations,  we denote by  $C$  some generic constants, that may vary from line to line and depend only on  the number   $ C_0$ in   \eqref{trin}.   Note these generic constants $C$ as below are independent of   $k$.

 In view of \eqref{trb}, it follows from
  Leibniz formula that
\begin{equation*}
	H^{k}{\Gamma}(f,f)=\sum\limits_{j=0}^{ k }\sum\limits_{p=0}^{  j }\binom{k}{j}\binom{j}{p}\mathcal{T}(H^{k-j}f, H^{j-p}f, H^{p}\mu^{1/2}).
\end{equation*}
As a result, taking the partial Fourier transform for $x$ variable on both sides  and using the  notation \eqref{Foutri}, we conclude
\begin{equation}\label{uppboun}
\int_{\mathbb{Z}^3}\left(\int_{0}^{T}\big|\big(\mathcal{F}_x(H^{k}\Gamma(f,f)), \widehat{H^{k}f}\big)_{L^2_v}\big|dt\right)^{1/2}d\Sigma(m) \leq \mathcal{J}_1+\mathcal{J}_2+\mathcal{J}_3,
\end{equation}
with
\begin{equation}\label{J1j2j3}
\left\{
\begin{aligned}
	\mathcal{J}_1&=\int_{\mathbb{Z}^3}\bigg[\sum\limits_{0 \leq p \leq k }\binom{k}{p}\int_{0}^{T}\big|\big (\hat{\mathcal{T}}(\hat{f},\ \widehat{H^{k-p}f},\ H^{p}\mu^{1/2}), \ \widehat{H^{k}f}\big)_{L^2_v}\big|dt\bigg]^{1\over2}d\Sigma(m),\\
	\mathcal{J}_2&=\int_{\mathbb{Z}^3} \bigg[\sum\limits_{j=1}^{k-1}\sum\limits_{p=0}^j\binom{k}{j}\binom{j}{p}\int_{0}^{T}\big|\big(\hat{\mathcal{T}}(\widehat{H^{k-j}f}, \widehat{H^{j-p}f}, H^{p}\mu^{1\over2}), \widehat{H^{k}f}\big)_{L^2_v}\big|dt\bigg]^{1\over2}d\Sigma(m),\\
	\mathcal{J}_3&=\int_{\mathbb{Z}^3}\left(\int_{0}^{T}\big|\big (\hat{\mathcal T}(\widehat{H^{k}f},\ \hat{f}, \ \mu^{1/2}),\  \widehat{H^{k}f}\big)_{L^2_v}\big|dt\right)^{1/2}d\Sigma(m).
\end{aligned}
\right.
\end{equation}
We proceed to estimate $ \mathcal J_1,\mathcal J_2,$ and $\mathcal J_3$ as follows.

 \smallskip
\noindent\underline{\it Estimate on $\mathcal J_1$}.
We  first control the  term   $\mathcal{J}_1$
by dividing it  into two terms. That is
\begin{equation}\label{j1}
\begin{aligned}
\mathcal{J}_1 & \leq   \int_{\mathbb{Z}^3}\left(\int_{0}^{T}\big|\big(\hat{\mathcal T}(\hat{f}, \widehat{H^{k}f}, \mu^{1/2}), \widehat{H^{k}f}\big)_{L^2_v}\big|dt\right)^{1\over 2}d\Sigma(m)
\\ & \quad +
\int_{\mathbb{Z}^3}\bigg[\sum\limits_{1 \leq p \leq k }\binom{k}{p}\int_{0}^{T}\big|\big(\hat{\mathcal{T}}(\hat{f}, \widehat{H^{k-p}f}, H^{p}\mu^{1/2}), \widehat{H^{k}f}\big)_{L^2_v}\big|dt\bigg]^{1\over2}d\Sigma(m)\\
&:=\mathcal J_{1,1}+\mathcal J_{1,2}.
\end{aligned}
\end{equation}
Direct verification shows that
\begin{equation}\label{mu}
\forall \ p\geq 0,~ \forall\ t\in[0,T], \quad \big|H^{p}  \mu^{1/2}\big|=\big|t^{\delta p}\partial_{v_1}^{p}  \mu^{1/2}\big| \leq
(2T^\delta)^pp!  \mu^{1/4}.
\end{equation}
Then we apply \eqref{ketres} with $\omega=H^p\mu^{1/2}$ to control $\mathcal J_{1,2}$ in \eqref{j1}  as follows: for any $\eps>0,$
\begin{equation}\label{tecest1}
	\begin{aligned}
		&\mathcal J_{1,2} =\int_{\mathbb{Z}^3}\bigg[\sum\limits_{1 \leq p \leq k }\binom{k}{p}\int_{0}^{T}\big|\big(\hat{\mathcal{T}}(\hat{f}, \widehat{H^{k-p}f}, H^{p}\mu^{1/2}), \widehat{H^{k}f}\big)_{L^2_v}\big|dt\bigg]^{1\over2}d\Sigma(m)\\
		&\leq C \int_{\mathbb{Z}^3}\bigg[\sum\limits_{p=1}^k\binom{k}{p}(2T^\delta)^pp! \\
		&\qquad\qquad \quad \times\int_{0}^{T} \bigg(
		 \int_{\mathbb Z_\ell^3 } \norm{\hat f(m-\ell)}_{L^2_v}     \normm{\widehat{H^{k-p}f}(\ell)} d\Sigma(\ell)\bigg)  \normm{\widehat{H^{k}f}  (m)}     dt\bigg]^{1\over2}d\Sigma(m)\\
		&\leq  \eps   \int_{\mathbb{Z}^3}\bigg[ \int_{0}^{T} \normm{\widehat{H^{k}f}  (t,m)}^2     dt \bigg]^{1\over2}d\Sigma(m)\\
		&+\frac{C}{ \eps}  \int_{\mathbb{Z}^3} \bigg[\int_{0}^{T} \bigg\{ \sum\limits_{p=1}^k\binom{k}{p}(2T^\delta)^pp! \int_{\mathbb Z^3 } \norm{\hat f(m-\ell)}_{L^2_v}     \normm{\widehat{H^{k-p}f}(\ell)}   d\Sigma(\ell) \bigg\}^2dt\bigg]^{1\over2}d\Sigma(m).
	\end{aligned}
\end{equation}
Moreover, in order to treat the last term in \eqref{tecest1} we apply \eqref{MF}  to get
\begin{equation}\label{tecest2}
	\begin{aligned}
		&\int_{\mathbb{Z}^3}\bigg[\int_{0}^{T} \bigg\{ \sum\limits_{p=1}^k\binom{k}{p}(2T^\delta)^{p}p!\int_{\mathbb Z_\ell^3 } \norm{\hat f(m-\ell)}_{L^2_v}     \normm{\widehat{H^{k-p}f}(\ell)}   d\Sigma(\ell) \bigg\}^2dt\bigg]^{1\over2}d\Sigma(m)\\
	&	\leq C\sum\limits_{p=1}^k\binom{k}{p}(2T^\delta)^{p}p!
	 \int_{\mathbb{Z}^3} \sup\limits_{0 < t \leq T}\norm{\hat{f}(t,m)}_{L^2_v}d\Sigma(m) \\
	&\qquad\qquad\qquad\qquad\qquad\qquad\qquad\qquad\times \int_{\mathbb{Z}^3}\Big[\int_{0}^{T}\normm{\widehat{H^{k-p}f}(t,m)} ^{2}dt \Big]^{1\over2}d\Sigma(m)\\
	&	\leq C  \eps_0^2   \sum\limits_{p=1}^k \frac{k! }{(k-p)!} (2T^\delta)^{p}  \frac{C_*^{k -p }[(k-p)!]^\tau }{(k-p+1)^2}\leq C  \eps_0^2C_*^{ k } (k!)^\tau \sum\limits_{p=1}^k   \frac{2^{-p}}{(k-p+1)^2},
	\end{aligned}
\end{equation}
where in the last line we used   condition  \eqref{suppose} as well as the fact that $	C_*>4T^\delta$.  For the last term in \eqref{tecest2} we  have,  denoting by $[k/2]$ the largest integer less than or equal to $k/2$,
\begin{equation}\label{teccom}
\begin{aligned}
	 & \sum\limits_{p=1}^k   \frac{2^{ -p }}{(k-p+1)^2}   \leq        \sum\limits_{p=1}^{[k/2] }\frac{1}{(k-p+1)^2}2^{-p}+\sum\limits_{p=[k/2]+1 }^{   k }\frac{1}{(k-p+1)^2}2^{-p}   \\
&\leq   C    \bigg\{\sum\limits_{p=1}^{ [k/2] }\frac{1}{(k+1)^2}2^{-p} +\sum\limits_{p=[k/2]+1}^{   k}\frac{1}{(k+1)^2}(k+1)^22^{-p} \bigg\} \leq     \frac{C}{(k+1)^2},	 \end{aligned}
\end{equation}
the last inequality using the fact  that
\begin{equation*}
\sum\limits_{p=[k/2]+1}^{ k }(k+1)^22^{-p} \leq\sum\limits_{p=[k/2]+1}^{ k }(k+1)^22^{-\frac{k}{2}}   \leq (k+1)^3 2^{-\frac{k}{2}}   \leq C.
\end{equation*}
As a result, we substitute \eqref{teccom} into  \eqref{tecest2} to conclude that
\begin{align*}
	&\int_{\mathbb{Z}^3}\bigg[\int_{0}^{T} \bigg\{ \sum\limits_{p=1}^k\binom{k}{p}(2T^\delta)^{p}p!\int_{\mathbb Z_\ell^3 } \norm{\hat f(m-\ell)}_{L^2_v}     \normm{\widehat{H^{k-p}f}(\ell)}   d\Sigma(\ell) \bigg\}^2dt\bigg]^{1\over2}d\Sigma(m)\\
	&\leq  C  \eps_0^2\frac{C_*^{ k } (k!)^\tau}{(k+1)^2},
\end{align*}
 which  with \eqref{tecest1}  yields that
\begin{equation}\label{j12}
	\mathcal J_{1,2}\leq  \eps   \int_{\mathbb{Z}^3}\bigg[ \int_{0}^{T} \normm{\widehat{H^{k}f}  (t,m)}^2     dt \bigg]^{1\over2}d\Sigma(m)+C\eps^{-1} \eps_0^2\frac{C_{*}^{k}(k!)^\tau}{(k+1)^2}.
\end{equation}
Moreover, following a similar  argument as above with slight modification,    we conclude that
 \begin{equation*}
 \begin{aligned}
 	&\mathcal J_{1,1}=\int_{\mathbb{Z}^3}\left(\int_{0}^{T}\big|\big(\hat{\mathcal T}(\hat{f}, \widehat{H^{k}f}, \mu^{1/2}), \widehat{H^{k}f}\big)_{L^2_v}\big|dt\right)^{1\over 2}d\Sigma(m)\\
 	&\leq  \eps    \int_{\mathbb{Z}^3}\bigg[ \int_{0}^{T} \normm{\widehat{H^{k}f}  (t,m)}^2     dt \bigg]^{1\over2}d\Sigma(m)\\
 	&\quad+C\eps^{-1}   \Big(\int_{\mathbb Z^3} \sup\limits_{0 < t \leq T}\norm{\hat{f}(t,m)}_{L^2_v}d\Sigma(m)\Big)\int_{\mathbb{Z}^3}\Big[\int_{0}^{T}\normm{\widehat{H^{k}f}(t,m)} ^{2}dt \Big]^{1\over2}d\Sigma(m) \\
 	&\leq \inner{ \eps +C\eps^{-1}\eps_0}   \int_{\mathbb{Z}^3}\bigg[ \int_{0}^{T} \normm{\widehat{H^{k}f}  (t,m)}^2     dt \bigg]^{1\over2}d\Sigma(m).
 	\end{aligned}
 \end{equation*}
 Here we used assumption  \eqref{apbound} to ensure the right-hand side is finite. 
Substitute the above estimate and \eqref{j12} into \eqref{j1} yields  that, for any $\eps>0$,
\begin{equation}\label{uppj1}
	\mathcal J_1\leq  \inner{ \eps +C\eps^{-1}\eps_0}   \int_{\mathbb{Z}^3}\bigg[ \int_{0}^{T} \normm{\widehat{H^{k}f}  (t,m)}^2     dt \bigg]^{1\over2}d\Sigma(m)+  C\eps^{-1}\eps_0^2\frac{C_{*}^{k}(k!)^\tau}{(k+1)^2}.
\end{equation}

 \smallskip
\noindent
\underline{\it Estimate on $\mathcal J_2$}.  Recall  $\mathcal{J}_2$ is given in \eqref{J1j2j3}.  Following a similar argument as that in \eqref{tecest1} and \eqref{tecest2} yields that, for any $\eps>0,$
\begin{equation}\label{j2+}
\begin{aligned}
\mathcal{J}_2 \leq& \eps\int_{\mathbb{Z}^3}\left(\int_{0}^{T}\normm{ \widehat{H^{k}f}(t,m)}^{2}dt\right)^{1/2}d\Sigma(m)\\&+
  C\eps^{-1} \sum\limits_{j=1}^{  k-1 }\sum\limits_{p=0}^{  j }\binom{k}{j}\binom{j}{p}(2T^\delta)^p p!\int_{\mathbb{Z}^3}\sup\limits_{0 < t \leq T} \norm{\widehat{H^{k-j}f}(t,m)}_{L^2_v}d\Sigma(m)\\&\qquad\qquad\qquad\qquad\quad\qquad\qquad\quad\times\int_{\mathbb{Z}^3}\left(\int_{0}^{T}\normm{ \widehat{H^{j-p}f}(t,m)}^{2}dt\right)^{1\over2}d\Sigma(m).
	\end{aligned}
\end{equation}
Moreover, we use  assumption \eqref{suppose} and then repeat
  the computation in \eqref{teccom}, to conclude that,  for any $1\leq j \leq k-1$,
\begin{multline*}
		 \sum\limits_{p=0}^{ j }\binom{j}{p}(2T^\delta)^p p!\int_{\mathbb{Z}^3}\left(\int_{0}^{T}\normm{\widehat{H^{j-p}f}(t,m)}^{2}dt\right)^{1/2}d\Sigma(m)\\
		  \leq  \eps_0 C_*^{j} (j!)^\tau \sum\limits_{p=0}^{ j }  \frac{2^{-p} }{(j-p+1)^2}
		\leq  C\eps_0\frac{C_{*}^{j}(j!)^\tau }{(j+1)^2}.
	\end{multline*}
Substituting the above estimate into the last term on the right-hand side of \eqref{j2+} and using again   condition \eqref{suppose}, we compute
\begin{equation*}
	\begin{aligned}
		&\sum\limits_{j=1}^{  k-1 }\sum\limits_{p=0}^{  j }\binom{k}{j}\binom{j}{p}{(2T^\delta)}^p p!\int_{\mathbb{Z}^3}\sup\limits_{0 < t \leq T} \norm{\widehat{H^{k-j}f}(t,m)}_{L^2_v}d\Sigma(m)\\&\qquad\qquad\qquad\qquad\quad\qquad\qquad\quad\times\int_{\mathbb{Z}^3}\left(\int_{0}^{T}\normm{ \widehat{H^{j-p}f}(t,m)}^{2}dt\right)^{1\over2}d\Sigma(m)\\
		&\leq C\eps_0 \sum\limits_{j=1}^{  k-1 }  \frac{k!}{j!(k-j)!} \frac{C_{*}^{j}(j!)^\tau }{(j+1)^2}\int_{\mathbb{Z}^3}\sup\limits_{0 < t \leq T} \norm{\widehat{H^{k-j}f}(t,m)}_{L^2_v}d\Sigma(m)  \\
		& \leq   C\eps_0^2 \sum\limits_{j=1}^{  k-1 }\frac{k!}{j!(k-j)!}\frac{C_{*}^{j}(j!)^\tau}{(j+1)^2} \frac{C_{*}^{k-j}[(k-j)!]^\tau}{(k-j+1)^2}\\
		&\leq    C\eps_0^2 C_{*}^{k}  \sum\limits_{j=1}^{  k-1 } \frac{k! (j!)^{\tau-1}[(k-j )!]^{\tau-1}}{(k-j+1)^2 (j+1)^2}
		\leq   C\eps_0^2 \frac{C_{*}^{k}(k!)^\tau }{(k+1)^2} ,
	\end{aligned}
\end{equation*}
the last inequality using the facts that $p!q!\leq (p+q)!$ and that $\tau\geq 1$.
This, together  with \eqref{j2+},  yields
\begin{equation}\label{uppj2}
\mathcal{J}_2 \leq  \eps \int_{\mathbb{Z}^3}\left(\int_{0}^{T}\normm{\widehat{H^{k}f}(t,m)}^{2}dt\right)^{1/2}d\Sigma(m)+	C \eps^{-1} \eps_0^2\frac{C_{*}^{k} (k!)^\tau}{(k+1)^2}.
\end{equation}

\smallskip
\noindent
\underline{\it Estimate on  $\mathcal J_3$}. It remains to deal with $\mathcal J_3$,  recalling $\mathcal{J}_3$ is given in \eqref{J1j2j3}.  We repeat the computation in \eqref{tecest1} and \eqref{tecest2} to conclude that
\begin{equation*}
\begin{aligned}
& \mathcal{J}_3  =\int_{\mathbb{Z}^3}\left(\int_{0}^{T}\big|\big (\hat{\mathcal T}(\widehat{H^{k}f},\ \hat{f}, \ \mu^{1/2}), \widehat{H^{k}f}\big)_{L^2_v}\big|dt\right)^{1/2}d\Sigma(m)\\
& \leq   \eps   \int_{\mathbb{Z}^3}\bigg[ \int_{0}^{T} \normm{\widehat{H^{k}f}  (t,m)}^2     dt \bigg]^{1\over2}d\Sigma(m)
		\\
		& \quad  +C\eps^{-1}\Big( \int_{\mathbb{Z}^3} \sup\limits_{0 < t \leq T}\norm{\widehat{H^{k}f}(t,m)}_{L^2_v}d\Sigma(m)\Big)  \int_{\mathbb{Z}^3}\Big[\int_{0}^{T}\normm{\hat f(t,m)} ^{2}dt \Big]^{1\over2}d\Sigma(m)\\
& \leq   \eps   \int_{\mathbb{Z}^3}\bigg[ \int_{0}^{T} \normm{\widehat{H^{k}f}  (t,m)}^2     dt \bigg]^{1\over2}d\Sigma(m)
		   +C \eps^{-1} \eps_0 \int_{\mathbb{Z}^3} \sup\limits_{0 < t \leq T}\norm{\widehat{H^{k}f}(t,m)}_{L^2_v}d\Sigma(m).
\end{aligned}
\end{equation*}
Combining the  upper bound of $\mathcal J_3$ above  and   estimates \eqref{uppj1} and \eqref{uppj2} with \eqref{uppboun}, we obtain the assertion in Proposition \ref{lemgamma}. The proof is completed.
 \end{proof}

\begin{proposition}\label{lem:com}
 Under  the same assumption as in   Proposition \ref{lemgamma}, we  can find  a constant $C$,  depending only on $T,\delta$ and the number  $ C_0$ in   \eqref{trisole} and \eqref{tretmate} but independent of $k$,  such that for any $\eps >0$,
	\begin{equation*}
	\begin{aligned}
	&\int_{\mathbb{Z}^3}\left(\int_{0}^{T}\big|\big(\mathcal{F}_x([H_\delta^{k},\  \mathcal{L}]f), \ \widehat{H_\delta^{k}f}\big)_{L^2_v}\big|dt\right)^{1/2}d\Sigma(m)
	\\ & \leq   \eps\int_{\mathbb{Z}^3}\left(\int_{0}^{T}\normm{ \widehat{H_\delta^{k}f}(t,m)}^{2}dt\right)^{1/2}d\Sigma(m)+C{\eps}^{-1}\frac{\eps_0C_{*}^{k-1}(k!)^\tau}{(k+1)^2}.
	\end{aligned}
	\end{equation*}
\end{proposition}

\begin{proof}
This is just a specific case of Proposition \ref{lemgamma}.
	Recall $\mathcal{L} $ is defined in \eqref{linearbol}, that is,
	\begin{eqnarray*}
	\begin{aligned}
	\mathcal L f&=-\Gamma( \mu^{1/2},f)-\Gamma(f,\mu^{1/2}) =-\mathcal T(\mu^{1/2}, f, \mu^{1/2})-\mathcal T(f,\mu^{1/2}, \mu^{1/2}).		\end{aligned}
	\end{eqnarray*}
 Then using again Leibniz's formula gives  that, denoting $H=H_\delta$,
 \begin{equation}\label{lastlabe}
\begin{aligned}
[H^{k},\ \mathcal{L}  ]f=&-\sum\limits_{j=1}^{ k }\sum\limits_{p=0}^{ j } \binom{k}{j} \binom{j}{p}\mathcal{T}(H^{j-p}\mu^{1/2}, H^{k-j}f, H^{p}\mu^{1/2})\\
&\qquad-\sum\limits_{j=1}^{  k }\sum\limits_{p=0}^{ j } \binom{k}{j} \binom{j}{p}\mathcal{T}(H^{k-j}f, H^{j-p}\mu^{1/2}, H^{p}\mu^{1/2})\\
 \stackrel{\rm def}{=}& R_1(f)+R_2(f).
   \end{aligned}
   \end{equation}
Moreover,  we may write,   as in \eqref{uppboun}, 
\begin{equation}\label{r2est}
\begin{aligned}
&	\int_{\mathbb{Z}^3}\left(\int_{0}^{T}\big|\big(\mathcal{F}_x\big (R_2(f)\big ), \widehat{H^{k}f}\big)_{L^2_v}\big|dt\right)^{1/2}d\Sigma(m)\\
& \leq \int_{\mathbb{Z}^3} \bigg[\sum\limits_{j=1}^{k}\sum\limits_{p=0}^j\binom{k}{j}\binom{j}{p}\int_{0}^{T}\big|\big(\hat{\mathcal{T}}(\widehat{H^{k-j}f},  H^{j-p}\mu^{1\over2}, H^{p}\mu^{1\over2}), \widehat{H^{k}f}\big)_{L^2_v}\big|dt\bigg]^{1\over2}d\Sigma(m).
\end{aligned}
\end{equation}
 By direct verification, it follows  that, for any $ p\geq 0$ and any $   t\in[0,T]$, 
\begin{multline*}
\quad \big| (1+|v|^{2+\gamma}-\Delta_v) H^{p}  \mu^{1/2}\big|=\big|(1+|v|^{2+\gamma}-\Delta_v) t^{\delta p}\partial_{v_1}^{p}  \mu^{1/2}\big| \leq C
(2T^\delta)^pp!  \mu^{1/8}.
\end{multline*}
This, with \eqref{mu}, enables to use  \eqref{trisole}  with $g=\mu^{1/2}$ and $\omega=H^{p}\mu^{1/2}$, to compute
\begin{align*}
	&\sum\limits_{j=1}^{k}\sum\limits_{p=0}^j\binom{k}{j}\binom{j}{p} \big|\big(\hat{\mathcal{T}}(\widehat{H^{k-j}f},  H^{j-p}\mu^{1\over2}, H^{p}\mu^{1\over2}), \widehat{H^{k}f}\big)_{L^2_v}\big| \\
	&\leq C \sum \limits_{j=1}^{k}\sum\limits_{p=0}^j\binom{k}{j}\binom{j}{p}(2T^\delta)^pp! \times\big[ (2T^\delta)^{j-p}(j-p)! \big]  		 \norm{\widehat{H^{k-j}f}(m)}_{L^2_v}        \normm{\widehat{H^{k}f}  (m)}  \\
	&\leq  C\sum \limits_{j=1}^{k}\frac{k!}{(k-j)!} (j+1)(2T^\delta)^j \norm{\widehat{H^{k-j}f}(m)}_{L^2_v}        \normm{\widehat{H^{k}f}  (m)}.    
	\end{align*}
Thus,  for any $\eps>0,$
\begin{equation*}
		\begin{aligned}
		&\int_{\mathbb{Z}^3} \bigg[\sum\limits_{j=1}^{k}\sum\limits_{p=0}^j\binom{k}{j}\binom{j}{p}\int_{0}^{T}\big|\big(\hat{\mathcal{T}}(\widehat{H^{k-j}f},  H^{j-p}\mu^{1\over2}, H^{p}\mu^{1\over2}), \widehat{H^{k}f}\big)_{L^2_v}\big|dt\bigg]^{1\over2}d\Sigma(m)\\
		&\leq C \int_{\mathbb{Z}^3}\bigg[  \int_{0}^{T} 		    \sum \limits_{j=1}^{k}\frac{k!}{(k-j)!} (j+1)(2T^\delta)^j \norm{\widehat{H^{k-j}f}(m)}_{L^2_v}        \normm{\widehat{H^{k}f}  (m)}     dt\bigg]^{1\over2}d\Sigma(m)\\
		&\leq  \eps   \int_{\mathbb{Z}^3}\bigg[ \int_{0}^{T} \normm{\widehat{H^{k}f}  (t,m)}^2     dt \bigg]^{1\over2}d\Sigma(m)\\
		&\quad+\frac{C}{ \eps}  \int_{\mathbb{Z}^3} \bigg[\int_{0}^{T} \bigg\{  \sum \limits_{j=1}^{k}\frac{k!}{(k-j)!} (j+1)(2T^\delta)^j \norm{\widehat{H^{k-j}f}(m)}_{L^2_v}       \bigg\}^2dt\bigg]^{1\over2}d\Sigma(m).
	\end{aligned}
\end{equation*}
As for the last term, we use  triangle inequality for  norms to get, recalling $C_*>4T^\delta$,
\begin{align*}
	& \int_{\mathbb{Z}^3} \bigg[\int_{0}^{T} \bigg\{  \sum \limits_{j=1}^{k}\frac{k!}{(k-j)!} (j+1)(2T^\delta)^j \norm{\widehat{H^{k-j}f}(m)}_{L^2_v}       \bigg\}^2dt\bigg]^{1\over2}d\Sigma(m)\\
	&\leq  \int_{\mathbb{Z}^3}  \sum \limits_{j=1}^{k}\frac{k!}{(k-j)!} (j+1)(2T^\delta)^j   \bigg[\int_{0}^{T}  \norm{\widehat{H^{k-j}f}(m)}_{L^2_v}     ^2dt\bigg]^{1\over2}d\Sigma(m)\\
	&\leq  T^{1/2}  \sum \limits_{j=1}^{k}\frac{k!}{(k-j)!} (j+1)(2T^\delta)^j    \int_{\mathbb{Z}^3}    \sup_{0<t\leq T}  \norm{\widehat{H^{k-j}f}(t,m)}_{L^2_v}      d\Sigma(m)\\
	&\leq 2 \eps_0 T^{1/2}   T^\delta \sum \limits_{j=1}^{k}\frac{k!}{(k-j)!} (j+1)(2T^\delta)^{j-1}   \frac{C_*^{k-j}[(k-j)!]^\tau}{(k-j+1)^2} \leq C  \eps_0\frac{C_*^{ k-1 } (k!)^\tau}{(k+1)^2},
\end{align*}
the last line using   inductive assumption \eqref{suppose} and the last inequality following from a similar argument as that in  \eqref{tecest2} and \eqref{teccom}.  Combining the above estimates we conclude that
\begin{multline*}
	\int_{\mathbb{Z}^3} \bigg[\sum\limits_{j=1}^{k}\sum\limits_{p=0}^j\binom{k}{j}\binom{j}{p}\int_{0}^{T}\big|\big(\hat{\mathcal{T}}(\widehat{H^{k-j}f},  H^{j-p}\mu^{1\over2}, H^{p}\mu^{1\over2}), \widehat{H^{k}f}\big)_{L^2_v}\big|dt\bigg]^{1\over2}d\Sigma(m)\\
	\leq \eps   \int_{\mathbb{Z}^3}\bigg[ \int_{0}^{T} \normm{\widehat{H^{k}f}  (t,m)}^2     dt \bigg]^{1\over2}d\Sigma(m)+ C{\eps}^{-1}\frac{\eps_0C_{*}^{k-1}(k!)^\tau}{(k+1)^2},
\end{multline*}
which with \eqref{r2est} yields
\begin{multline*}
	\int_{\mathbb{Z}^3}\left(\int_{0}^{T}\big|\big(\mathcal{F}_x\big (R_2(f)\big ), \widehat{H^{k}f}\big)_{L^2_v}\big|dt\right)^{1/2}d\Sigma(m)\\
	\leq \eps   \int_{\mathbb{Z}^3}\bigg[ \int_{0}^{T} \normm{\widehat{H^{k}f}  (t,m)}^2     dt \bigg]^{1\over2}d\Sigma(m)+ C{\eps}^{-1}\frac{\eps_0C_{*}^{k-1}(k!)^\tau}{(k+1)^2}. 
\end{multline*}
Similarly,  using \eqref{tretmate} instead of \eqref{trisole}, we can verify that the above estimate still holds true with $R_2(f)$  replaced by $R_1(f)$.  Thus the assertion in Proposition \ref{lem:com} follows by observing 
\begin{multline*}
	\int_{\mathbb{Z}^3}\left(\int_{0}^{T}\big|\big(\mathcal{F}_x([H_\delta^{k},\  \mathcal{L}]f), \ \widehat{H_\delta^{k}f}\big)_{L^2_v}\big|dt\right)^{1/2}d\Sigma(m)\\\leq \sum_{j=1}^2 \int_{\mathbb{Z}^3}\left(\int_{0}^{T}\big|\big(\mathcal{F}_x\big (R_j(f)\big ), \widehat{H^{k}f}\big)_{L^2_v}\big|dt\right)^{1/2}d\Sigma(m)
\end{multline*}
due to \eqref{lastlabe}.    The proof is thus completed.
\end{proof}

\section{Analytic smoothing effect for strong angular singularity}
\label{sec: strong}
In this section we consider the case when the cross-section has strong angular singularity, that is,  the number  $s$ in \eqref{angu} satisfies that   $1/2\leq s<1$. This will yield the analytic regularity of weak solutions to the Boltzmann equation \eqref{eqforper} at any positive time.

  \subsection{Quantitative estimate for  directional derivatives}
To get the analyticity of solutions at positive times, it relies on a   crucial estimate on the derivatives in the direction $H_\delta$ defined in \eqref{vecM} with $\delta$ therein satisfying condition \eqref{relation}.    In this part, we will perform an energy estimate on the directional derivatives of regular solutions, and the treatment for the classical derivatives will be presented in the next subsection.

 \begin{theorem}\label{theorem} Assume that the cross-section satisfies \eqref{kern} and \eqref{angu} with $   \gamma \geq 0$ and $1/2\leq  s <1$.   Let $T\geq 1$ be arbitrarily given, and
 	 let $f\in L_m^1L_T^\infty L_v^2$  be any solution to the Cauchy problem \eqref{eqforper} satisfying that,
 	for any $N\in\mathbb Z_+$ and any $\beta\in\mathbb Z_+^3,$
 \begin{equation}\label{higher+}
 \begin{aligned}
    &\int_{{\mathbb Z}^3}  \Big( \sup_{0<t\leq T}t^{ \frac{1+2s}{2s}( N+\abs\beta)}   \abs{m}^{N}\norm{ {\partial_v^\beta \hat f (t,m,\cdot)}}_{L^2_v}\Big) d \Sigma(m)  \\
    & \qquad \qquad +\int_{\mathbb Z^3}  \Big(  \int_0^T t^{ \frac{1+2s}{s}(N+\abs\beta) }\abs{m}^{2N}  \normm{ {\partial_v^\beta\hat f (t,m,\cdot )}} ^2dt\Big)^{1/2} d \Sigma(m)
<+\infty.
  %\norm{f_0}_{ L^1_kL^2_v}.
\end{aligned}
\end{equation}
Moreover, let $H_\delta$ be defined by \eqref{vecM}  with $\delta$ an  arbitrarily given number  satisfying \eqref{relation}.   Then  there exists a sufficiently small constant $\eps_0>0$ and a large constant $L\geq 1$, with $L$ depending only on $T,\delta$,  and the numbers $c_0, C_0$ in Section \ref{sec:prelim},  such that if
\begin{equation}\label{lower++}
 \int_{{\mathbb Z}^3}  \Big( \sup_{0<t\leq T}   \norm{ {\hat f (t, m )}}_{L^2_v}\Big) d \Sigma(m)+\int_{{\mathbb Z}^3}   \Big( \int_0^T  \normm{ \hat f (t,m )} ^2dt \Big)^{\frac{1}{2}} d \Sigma(m) \leq \eps_0,
 \end{equation}
   then
  the  estimate
 	\begin{multline}\label{k}
 	\int_{\mathbb{Z}^3}\sup\limits_{0 < t \leq T}\norm{\widehat{H_\delta^{k}f}(t,m)}_{L^2_v}d\Sigma(m)\\
 	+ \int_{\mathbb{Z}^3}\left(\int_{0}^{T}\normm{  \widehat{H_\delta^{k}f}(t,m)}^{2}dt\right)^{1/2}d\Sigma(m) \leq \frac{\eps_0 L^{k }  k!  }{(k+1)^2}
 	\end{multline}
 	holds true for any $k\in\mathbb Z_+$.
 	Moreover, the above estimate \eqref{k} is still true if we replace $H_\delta$ by
 	$$\frac{1}{1+\delta} t^{\delta+1} \partial_{x_i}+t^\delta \partial_{v_i}$$
 	with  $i=2 $ or $3$.
\end{theorem}

\begin{proof}
  To simplify the notations, we will use the capital letter $C$ to denote some generic constants, that may vary from line to line and depend only on $T,\delta,$ and the numbers $c_0,C_0$ in   Section \ref{sec:prelim}.   Note these generic constants $C$ as below are independent of the derivative order denoted by $k$.  If there is no confusion, in the following argument we will write $H=H_\delta$ for short, omitting the subscript $\delta$.

 We  use induction on $k$ to prove  the quantitative estimate \eqref{k}. The validity of \eqref{k} for $k=0$ follows from \eqref{lower++} if we choose $L\geq 1$.  Using the notation $H:=H_\delta$  and   supposing the estimate
\begin{multline}\label{suppose+}
\int_{\mathbb{Z}^3}\sup\limits_{0 < t \leq T}\norm{\widehat{H^{j}f}(t,m)}_{L^2_v}d\Sigma(m)\\+ \int_{\mathbb{Z}^3}\left(\int_{0}^{T}\normm{  \widehat{H^{j}f}(t,m)}^{2}dt\right)^{1/2}d\Sigma(m) \leq \frac{\eps_0L^{j}  j! }{(j+1)^2}
\end{multline}
holds true for any  $ j \leq k-1 $ with given $k\geq 1$, we will prove in the following argument that estimate \eqref{suppose+} still holds true for $j=k$ provided  $L\geq 4T^\delta$.   

To do so  we begin with the claim that  the estimate
\begin{equation}\label{aprik}
  \int_{\mathbb{Z}^3}\sup\limits_{0 < t \leq T}\norm{\widehat{H ^{k}f}(t,m)}_{L^2_v}d\Sigma(m)
  +      \int_{\mathbb{Z}^3}\Big (\int_{0}^{T}\normm{ \widehat{H^{k}f}(t,m)}^{2}dt\Big )^{1\over2}d\Sigma(m)
<+\infty 
\end{equation}
holds true for any $k\in\mathbb Z_+.$  In fact, by  Leibniz's formula we compute,     for any $0<t\leq T$ and for any $m\in \mathbb Z^3$,
\begin{align*}
	\normm{ \widehat{H^{k}f}(t,m)}&\leq C_{\delta,k} \sum_{j\leq k}  t^{(1+\delta)j+\delta (k-j) }\abs{m_1}^{j}\normm{\partial_{v_1}^{k-j}\hat f (t,m)}\\
	& \leq C_{\delta,k} \sum_{j\leq k} t^{   \delta k+j  }\abs{m}^{j}\normm{\partial_{v_1}^{k-j}\hat f (t,m)} \\
	& \leq C_{\delta,k} (1+T) ^{k}   t^{\big (\delta-\frac{1+2s}{2s}\big)k} \sum_{j\leq k} t^{ \frac{1+2s}{2s} k}\abs{m}^{j}\normm{\partial_{v_1}^{k-j}\hat f (t,m)},
\end{align*} 
and similarly, 
\begin{equation}\label{etj}
	\norm{ \widehat{H^{k}f}(t,m)}_{L^2_v} \leq C_{\delta,k} (1+T) ^{k}   t^{\big (\delta-\frac{1+2s}{2s}\big)k} \sum_{j\leq k} t^{ \frac{1+2s}{2s} k}\abs{m}^{j}\norm{\partial_{v_1}^{k-j}\hat f (t,m)}_{L^2_v}, 
\end{equation}  
where  $C_{\delta,k}$ is a constant depending only on $k$ and $\delta$. Then assertion \eqref{aprik} follows   from assumption  \eqref{higher+} by observing the fact that  $\delta>\frac{1+2s}{2s}.$ 

 {\it Step 1).}
Applying $H^{k}$ to   equation    \eqref{eqforper} yields
	\begin{equation*}
	\begin{aligned}
	\inner{\partial_{t} +v\cdot\partial_{x}  +\mathcal{L}} H^{k}  f &= -[H^{k}, \partial_{t}+v\cdot\partial_{x}] f - [ H^{k}, \ \mathcal{L}] f + H^{k}\Gamma(f, f)\\
	&= \delta kt^{\delta-1}\partial_{v_1} H^{k-1}f - [ H^{k}, \ \mathcal{L}] f + H^{k}\Gamma(f, f),
	\end{aligned}
	\end{equation*}
	the last equality using \eqref{kehigher}.  Furthermore, we perform
partial  Fourier transform  in $x$ and then consider the real part after taking the  inner product of  $L_v^2$ with $ \widehat{H^{k} f}$, to obtain
\begin{equation*}
\begin{aligned}
 & \frac{1}{2}\frac{d}{dt}\norm{\widehat{H^{k}f}}_{L^2_v}^2 +\big(\mathcal{L}\widehat{H^{k}f},\  \widehat{H^{k}f}\big)_{L^2_v}
  \\ &\leq  \delta k  t^{\delta-1} \big|\big( \partial_{v_1} \widehat{H^{k-1}f},\  \widehat{H^{k}f}\big)_{L^2_v}\big|+ \big|\big(\mathcal{F}_x([ H^{k},\ \mathcal{L}]f), \widehat{H^{k}f}\big)_{L^2_v}\big|
\\&\quad +\big|\big(\mathcal{F}_x(H^{k}\Gamma(f,f)), \ \widehat{H^{k}f}\big)_{L^2_v}\big|.
\end{aligned}
\end{equation*}
This with \eqref{rela} yields that
\begin{equation}\label{enestimate}
\begin{aligned}
&\frac{1}{2}\frac{d}{dt}\norm{\widehat{H^{k}f}}_{L^2_v}^2+c_0\normm{ \widehat{H^{k}f}}^{2}\\
&\leq \norm{\widehat{H^{k}f}}_{L^2_v}^{2}+\delta k  t^{\delta-1} \big|\big( \partial_{v_1} \widehat{H^{k-1}f},\  \widehat{H^{k}f}\big)_{L^2_v}\big|
\\&\quad+ \big|\big(\mathcal{F}_x([ H^{k},\ \mathcal{L}]f), \widehat{H^{k}f}\big)_{L^2_v}\big|+\big|\big(\mathcal{F}_x(H^{k}\Gamma(f,f)), \ \widehat{H^{k}f}\big)_{L^2_v}\big|.
\end{aligned}
\end{equation}
For the second term on the right-hand side of \eqref{enestimate},  it follows from \eqref{+lowoftri}  that,  recalling $1/2\leq s<1$,
\begin{multline*}
 \delta k  t^{\delta-1} \big|\big( \partial_{v_1} \widehat{H^{k-1}f},\  \widehat{H^{k}f}\big)_{L^2_v}\big|\leq Ck\norm{\widehat{H^{k-1}f}}_{H_v^{s}}\norm{\widehat{H^{k}f}}_{H_v^{s}}\\
	\leq \frac{c_0}{2} \normm{\widehat{H^{k}f}} ^2+ C k^2\normm{\widehat{H^{k-1}f}}^2.
\end{multline*}
Thus,
\begin{equation}\label{hkenergy+}
\begin{aligned}
  \frac{1}{2}\frac{d}{dt}\norm{\widehat{H^{k}f}}_{L^2_v}^2+& \frac{c_0}{2}\normm{ \widehat{H^{k}f}}^{2}
  \leq \norm{\widehat{H^{k}f}}_{L^2_v}^{2}+C k^2\normm{\widehat{H^{k-1}f}}^2\\
 &+ \big|\big(\mathcal{F}_x([ H^{k},  \mathcal{L}]f), \widehat{H^{k}f}\big)_{L^2_v}\big|+\big|\big(\mathcal{F}_x(H^{k}\Gamma(f,f)),   \widehat{H^{k}f}\big)_{L^2_v}\big|.
\end{aligned}
\end{equation}
Together with Gronwall's inequality,  we integrate the above estimate  over $[0,t]$ for any $0<t\leq T$; this implies  that
\begin{equation*}\label{eneres}
\begin{aligned}
& \sup\limits_{0<t\leq T}\norm{\widehat{H^{k}f}(t)}_{L^2_v}^2+   \int_{0}^{T}\normm{ \widehat{H^{k}f}(t)}^{2}dt
\\
& \leq  C \lim_{t\rightarrow 0}\norm{\widehat{H^{k}f}(t)}_{L^2_v}^2+ C    k^2 \int_{0}^{T} \normm{\widehat{H^{k-1}f}}^2dt\\ &\quad+ C\int_{0}^{T}\big|\big(\mathcal{F}_x([ H^{k},\ \mathcal{L}]f), \widehat{H^{k}f}\big)_{L^2_v}\big|dt
+C\int_{0}^{T}\big|\big(\mathcal{F}_x(H^{k}\Gamma(f,f)), \ \widehat{H^{k}f}\big)_{L^2_v}\big|dt,
\end{aligned}
\end{equation*}
and thus
\begin{equation}\label{rt}
\begin{aligned}
& \int_{\mathbb Z^3} \sup\limits_{0<t\leq T}\norm{\widehat{H^{k}f}(t,m)}_{L^2_v}d\Sigma(m)+  \int_{\mathbb Z^3}\Big( \int_{0}^{T}\normm{ \widehat{H^{k}f}(t,m)}^{2}dt\Big)^{1/2}d\Sigma(m)
\\
& \leq  C \int_{\mathbb Z^3}  \Big[\lim_{t\rightarrow 0}\norm{\widehat{H^{k}f}(t,m)}_{L^2_v}^2\Big]^{1\over2}d\Sigma(m)+ C  k \int_{\mathbb Z^3} \Big( \int_{0}^{T}\normm{\widehat{H^{k-1}f}(t,m)}^2dt\Big)^{1\over2} d\Sigma(m)\\ &\quad+C \int_{\mathbb{Z}^3}\left(\int_{0}^{T}\big|\big(\mathcal{F}_x([H^{k},\  \mathcal{L}]f), \widehat{H^{k}f}\big)_{L^2_v}\big|dt\right)^{1/2}d\Sigma(m)\\
&\quad+C\int_{\mathbb{Z}^3}\left(\int_{0}^{T}\big|\big(\mathcal{F}_x(H^{k}\Gamma(f,f)), \widehat{H^{k}f}\big)_{L^2_v}\big|dt\right)^{1/2}d\Sigma(m).
\end{aligned}
\end{equation}
We will proceed to deal with the terms on the right-hand side of \eqref{rt}.

 {\it Step 2).} For the first term on  the right-hand side of \eqref{rt}, we claim
 \begin{equation}\label{frt}
 	\int_{\mathbb Z^3}  \Big[\lim_{t\rightarrow 0}\norm{\widehat{H^{k}f}(t,m)}_{L^2_v}^2\Big]^{1\over2}d\Sigma(m)=0.
 \end{equation}
 In fact, 
by \eqref{higher+},
 we see that for each $j\in\mathbb Z_+$,
 $$
 \sup_{0<t\leq T} t^{\frac{1+2s}{2s}k}  |m|^{j}\norm{\partial_{v_1}^{k-j}\hat f(t,m)}_{L_v^2}\in L_m^1,
$$
 which implies
  \begin{eqnarray*}
  \sup_{m\in\mathbb Z^3}\Big( \sup_{0<t\leq T} t^{\frac{1+2s}{2s}k}  |m|^{j}\norm{\partial_{v_1}^{k-j}\hat f(t,m)}_{L_v^2} \Big) \leq C_{k,j}<+\infty,
    \end{eqnarray*}
    where $C_{k,j}$ are   constants depending only on $k$ and $j.$
 As a result, combining the above estimate with \eqref{etj}  yields that, for any $0<t\leq T$ and any $m\in\mathbb Z^3,$
   \begin{equation*} 
  \norm{\widehat{H^{k}f}(t,m)}_{L_v^2} \leq C_{\delta,k}(1+T)^k t^{\big (\delta-\frac{1+2s}{2s}\big)k}  \sum_{j\leq k} C_{k,j},
  \end{equation*}
  recalling  $C_{\delta,k}$ a constant depending only on $k$ and $\delta$. This with condition \eqref{relation} yields
\begin{equation*}
 \forall\ m\in\mathbb Z^3,\quad 		\lim_{t\rightarrow 0} \norm{\widehat{H^{k}f}(t,m)}_{L_v^2}=0,
	\end{equation*}
and thus   assertion \eqref{frt} follows.
	
{\it Step 3).} For the second term on the right-hand side of \eqref{rt}, it follows from inductive assumption \eqref{suppose+} that
\begin{equation}\label{srt}
	k \int_{\mathbb Z^3} \Big( \int_{0}^{T}\normm{\widehat{H^{k-1}f}(t,m)}^2dt\Big)^{1\over2} d\Sigma(m)\leq \frac{\eps_0L^{k-1}k!}{k^2}\leq C\frac{\eps_0L^{k-1}k!}{(k+1)^2}
\end{equation}
for $k\geq 1.$  By assertion \eqref{aprik} which holds true for any $k\in\mathbb Z_+$,  we see condition \eqref{apbound} in Proposition \ref{lemgamma} is fulfilled. Moreover it follows from  inductive assumption \eqref{suppose+} that condition  \eqref{suppose} holds with $C_*=L$ therein.  This enables to  apply Propositions   \ref{lemgamma} and \ref{lem:com}  to control the remaining terms on the right-hand side of \eqref{rt}; this gives that the estimate
\begin{equation}\label{lart}
\begin{aligned}
	&  \int_{\mathbb{Z}^3}\left(\int_{0}^{T}\big|\big(\mathcal{F}_x([H^{k},\  \mathcal{L}]f), \widehat{H^{k}f}\big)_{L^2_v}\big|dt\right)^{1/2}d\Sigma(m)\\
&\qquad\qquad+ \int_{\mathbb{Z}^3}\left(\int_{0}^{T}\big|\big(\mathcal{F}_x(H^{k}\Gamma(f,f)), \widehat{H^{k}f}\big)_{L^2_v}\big|dt\right)^{1/2}d\Sigma(m)\\
&\leq   C{\eps}^{-1}\eps_0   \int_{\mathbb{Z}^3}\sup\limits_{0 < t \leq T}\norm{\widehat{H^{k}f}(t,m)}_{L^2_v}d\Sigma(m)\\
&\quad +    \inner{\eps+C{\eps}^{-1}\eps_0}  \int_{\mathbb{Z}^3}\Big (\int_{0}^{T}\normm{ \widehat{H^{k}f}(t,m)}^{2}dt\Big )^{1/2}d\Sigma(m) \\
&\quad+C \eps^{-1}  \eps_0^2 \frac{L^{k} k!}{(k+1)^2} +C \eps^{-1}     \frac{\eps_0L^{k-1} k!}{(k+1)^2}
\end{aligned}
\end{equation}
holds true for any $\eps>0$.

{\it Step 4).} Substituting   estimates \eqref{frt}, \eqref{srt} and \eqref{lart} into \eqref{rt}, we obtain   that, for any $\eps>0$
\begin{equation}\label{00}
\begin{aligned}
&\int _{Z^3}\sup\limits_{0 < t \leq T}\norm{\widehat{H^{k}f}(t,m)}_{L^2_v}d\Sigma (m)+  \int _{Z^3}\left(\int_{0}^{T} \normm{ \widehat{H^{k}f}(t,m)}^{2}dt\right)^{1/2}d\Sigma (m)
\\ & \leq   C{\eps}^{-1}\eps_0   \int_{\mathbb{Z}^3}\sup\limits_{0 < t \leq T}\norm{\widehat{H^{k}f}(t,m)}_{L^2_v}d\Sigma(m)\\
&\quad +    \inner{\eps+C{\eps}^{-1}\eps_0}  \int_{\mathbb{Z}^3}\Big (\int_{0}^{T}\normm{ \widehat{H^{k}f}(t,m)}^{2}dt\Big )^{1/2}d\Sigma(m) \\
&\quad+C \eps^{-1}  \eps_0^2 \frac{L^{k} k!}{(k+1)^2} +C \eps^{-1}     \frac{\eps_0L^{k-1} k!}{(k+1)^2}.
\end{aligned}
\end{equation}
 Note that $\eps_0$ is a sufficiently small number, so we may assume without loss of generality that
\begin{equation}\label{smass}
	C\eps_0\leq 1/16
\end{equation}
with $C>0$ the constant in \eqref{00}. Consequently,
  if we choose in particular $\eps=1/4$ in \eqref{00},  then in view of \eqref{smass} we have
  \begin{multline*}
   	\int _{Z^3}\sup\limits_{0 < t \leq T}\norm{\widehat{H^{k}f}(t,m)}_{L^2_v}d\Sigma (m)+  \int _{Z^3}\left(\int_{0}^{T} \normm{ \widehat{H^{k}f}(t,m)}^{2}dt\right)^{1/2}d\Sigma (m)
\\
 \leq    \frac12    \frac{ \eps_0  L^{k} k!}{(k+1)^2} +8C   \frac{ \eps_0  L^{k-1} k!}{(k+1)^2}  \leq  \frac{\eps_0L^{k} k!}{(k+1)^2},
  \end{multline*}
  provided $L$ is large enough such that $L\geq 16C.$     This yields the validity of \eqref{suppose+} for $j=k$.
  Thus assertion \eqref{k} follows. The treatment for   	$$\frac{1}{1+\delta} t^{\delta+1} \partial_{x_i}+t^\delta \partial_{v_i}, \quad i=2 \textrm{ or } 3,$$
  is just in the same way.
The proof of Theorem \ref{theorem} is completed.
  \end{proof}

\subsection{Proof of Theorem \ref{mainresult}: analytic regularization effect for $\boldsymbol{\frac12 \leq s<1}$} \label{subsec:ana}
Here  we   prove Theorem \ref{mainresult}    for the case of $1/2\leq s<1$,  and it suffices to prove that for any $T\geq 1$ and any $\lambda>1+\frac{1}{2s}$, there exists a constant $C$, depending only on $T,\lambda,$ and the numbers $c_0, C_0$ in \eqref{rela} and  \eqref{trin},  such that
\begin{equation}\label{anestimate+}
	\forall \ \alpha , \beta \in \mathbb{Z}_{+}^3, \quad 	\sup_{0<t\leq T}t^{(\lambda+1)\abs\alpha+ \lambda \abs\beta} \norm{\partial_x^{\alpha}\partial_{v}^{\beta}f(t)}_{L^2_{x,v}} \leq    C^{|\alpha|+|\beta|+1}(|\alpha|+|\beta|)! .
\end{equation}
The key part to prove \eqref{anestimate+} is the quantitative estimate \eqref{k}.   In the following discussion,  let $\delta_j, j=1,2,$ be defined in terms of $\lambda$ by \eqref{de1de2}.  Accordingly define  $H_{\delta_j}, j=1,2,$  by \eqref{vecM}.

    Under the smallness condition \eqref{smint},  Duan-Liu-Sakamoto-Strain \cite{MR4230064} obtained  the global existence and uniqueness of the mild solution $f\in L_m^1L_T^\infty L_v^2$  to  the Boltzmann equation \eqref{eqforper},  which satisfies that there exists a constant $C_1>0$   such that  for any $T\geq 1,$
\begin{equation}\label{conc1}
	 \int_{{\mathbb Z}^3}  \Big( \sup_{0<t\leq T}   \norm{ {\hat f (t, m )}}_{L^2_v}\Big) d \Sigma(m)+\int_{{\mathbb Z}^3}   \Big( \int_{0}^T  \normm{ \hat f (t,m )} ^2dt \Big)^{\frac{1}{2}} d \Sigma(m) \leq C_1 \epsilon.
\end{equation}
Moreover,  it is shown in \cite{MR4356815} that the above mild solution admits  Gevrey  regularity at $t>0,$ that is,  there exists a constant $C_2>0$ such that the estimate
\begin{equation*}
	\begin{aligned}
    &\int_{{\mathbb Z}^3}  \Big( \sup_{0<t\leq T} \phi(t)^{ \frac{1+2s}{2s}(N+\abs\beta)}  \abs{m}^{N}\norm{ {\partial_v^\beta \hat f (t,m)}}_{L^2_v}\Big) d \Sigma(m)  \\
    &\   +\int_{\mathbb Z^3}  \Big(\int_ {0}^T\phi(t)^{ \frac{1+2s}{s}(N+\abs\beta) }\abs{m}^{2N}  \normm{ {\partial_v^\beta\hat f (t,m)}} ^2dt\Big)^{1\over2} d \Sigma(m)
\le C_2^{N+\abs\beta+1}(N+\abs\beta)^{\frac{1+2s}{2s}}
\end{aligned}
\end{equation*}
holds true for  any $N\in \mathbb Z_+$ and any $\beta\in\mathbb Z_+^3$, where   $\phi(t)=\min\{t,\  1\}$. Note the constant  $C_1$ in \eqref{conc1} is independent of $T$ and the fact that 
\begin{align*}
\forall\ k\in\mathbb Z_+,\ \forall\ 0<t\leq T,\quad 	t^{ \frac{1+2s}{2s}k}\leq  T^{\frac{1+2s}{2s}k} \phi(t)^{ \frac{1+2s}{2s}k},
\end{align*}
and thus conditions \eqref{higher+} and \eqref{lower++}    are fulfilled by the above mild solution $f$, provided $\epsilon$ is small enough.
 This enables us to apply  Theorem \ref{theorem} to  $H_{\delta_j}, j=1,2,$    given above,  to conclude that  for any $T\geq 1,$ there exists
a constant $L$, depending only on $T,\delta_1,\delta_2,$ and the numbers $c_0, C_0$ in \eqref{rela} and \eqref{trin},  such that for  each $j=1,2$,  the  estimate
 	\begin{multline}\label{k+}
 	\int_{\mathbb{Z}^3}\sup\limits_{0 < t \leq T}\norm{\widehat{H_{\delta_j}^{k}f}(t,m)}_{L^2_v}d\Sigma(m)\\
 	+ \int_{\mathbb{Z}^3}\left(\int_{0}^{T}\normm{  \widehat{H_{\delta_j}^{k}f}(t,m)}^{2}dt\right)^{1/2}d\Sigma(m) \leq \frac{\eps_0 L^{k } k! }{(k+1)^2}
 	\end{multline}
 	holds true for  any $k\in\mathbb Z_+$,   where $\eps_0=C_1\epsilon$ with $C_1$ the constant in \eqref{conc1}.    Observe the discrete  Lebesgue spaces  $\ell^p$ are increasing in $p\in [1,+\infty]$, so that  in particular  $L_m^1\subset L_m^2$ for
 	$m\in\mathbb Z^3$. Then it follows from \eqref{k+}
 	that, for any $k\in\mathbb Z_+$ and each $j=1,2$,
 	\begin{multline}\label{k++}
 		\sup_{0<t\leq T}\norm{H_{\delta_j}^{k}f(t)}_{L_{x,v}^2}=   \sup_{0<t\leq T}\norm{\widehat{H_{\delta_j}^{k}f}(t)}_{L_{m}^2 L_v^2}\\
 		\leq   \int_{\mathbb{Z}^3}\sup\limits_{0 < t \leq T}\norm{\widehat{H_{\delta_j}^{k}f}(t,m)}_{L^2_v}d\Sigma(m)\leq \frac{\eps_0 L^{k } k! }{(k+1)^2}.
 	\end{multline}
 Next we will deduce the estimate on   classical derivatives. As a preliminary step, we first prove that,  for any $ k\in\mathbb Z_+$,
  \begin{equation}\label{pse1}
	\norm{(A_1+A_2)^k f}_{L^2_{x,v}}\leq 2^{k} \norm{ A_1^k f}_{L^2_{x,v}}+2^{k} \norm{ A_2^k f}_{L^2_{x,v}},
 \end{equation}
where $A_j, j=1,2,$ are  two Fourier multipliers  with symbols $ a_j=a_j(m,\eta)$, that is,
 \begin{eqnarray*}
 \mathcal F_{x,v}	(A_j f)(m,\eta)=a_j(m,\eta)\mathcal F_{x,v}f(m,\eta),
 \end{eqnarray*}
 with $\mathcal F_{x,v} f   $   the full Fourier transform in $(x,v)\in\mathbb T\times\mathbb R^3.$
 To prove \eqref{pse1} we compute
 \begin{align*}
&\Big|\mathcal F_{x,v}\big(	(A_1+A_2)^kf \big)(m,\eta)\Big|^2 =\big|  	\big (a_1  (m,\eta)+a_2 (m,\eta)\big )^k \mathcal F_{x,v}f (m,\eta)\big|^2  \\
&\leq \big(|a_1(m,\eta)|+|a_2(m,\eta)| \big)^{2k}  \times  \big| \mathcal{F}_{x,v}f(m, \eta) \big |^2 \\
&\leq 2^{2k} \big|a_1(m,\eta)^k \mathcal{F}_{x,v}f(m, \eta) \big|^2  +  2^{2k} \big| a_2(m,\eta)^k\mathcal{F}_{x,v}f( m, \eta) \big| ^{2}\\
&\leq 2^{2k}\big|  \mathcal{F}_{x,v}(A_1^kf)( m, \eta) \big|^2  + 2^{2k} \big|   \mathcal{F}_{x,v}(A_2^kf)(  m, \eta) \big| ^{2},
\end{align*}
the second inequality using the fact that $(p+q)^{ 2k}\leq (2p)^{ 2k}+(2q)^{ 2k}$ for any numbers $p,q\geq 0$ and any $k\in\mathbb Z_+$.
 As a result, we combine the above estimate with Parseval equality, to conclude that
   \begin{equation*}
   	\begin{aligned}
   		& \norm{(A_1+A_2)^kf }^2_{L^2_{x,v}} =\int_{\mathbb Z^3\times \mathbb R^3} \Big|\mathcal F_{x,v}\big(	(A_1+A_2)^kf \big)(m,\eta)\Big|^2d\Sigma(m)\, d\eta
   		\\ &\leq 2^{2k} \int_{\mathbb Z^3\times \mathbb R^3} \big|  \mathcal{F}_{x,v}(A_1^kf)(  m, \eta) \big|^2 d\Sigma(m)  d\eta + 2^{2k} \int_{\mathbb Z^3\times \mathbb R^3} \big|   \mathcal{F}_{x,v}(A_2^kf)( m, \eta) \big| ^{2} d\Sigma(m)  d\eta
   		\\& \leq     2^{2k}\norm{ A_1^{k}f }^2_{L^2_{x,v}}+2^{2k}\norm{ A_2^{k}f }^2_{L^2_{x,v}}.
   	\end{aligned}
   \end{equation*}
   This gives \eqref{pse1}.   Now we use \eqref{generate} and  then apply \eqref{pse1}  with
   \begin{align*}
   A_1=	\frac{(\delta_2+ 1)(\delta_1+1)}{\delta_2-\delta_1}  H_{\delta_1},\quad A_2=-\frac{(\delta_2+ 1)(\delta_1+1)}{\delta_2-\delta_1} t^{\delta_1-\delta_2}H_{\delta_2},
   \end{align*}
 to compute that, observing $\delta_1>\delta_2,$
  \begin{align*}
 	\sup_{0<t\leq T}t^{(\lambda+1)k} \norm{\partial_{x_1}^{k}f(t)}_{L^2_{x,v}} &=   \sup_{0<t\leq T} \norm{(A_1+A_2)^kf(t)}_{L^2_{x,v}}\\
 	& \leq 2^{k} \sup_{0<t\leq T} \norm{ A_1^kf(t)}_{L^2_{x,v}} +2^{k} \sup_{0<t\leq T} \norm{ A_2^kf(t)}_{L^2_{x,v}}\\ \
 	&\leq C_3^{k}  \sup_{0<t\leq T}\inner{\norm{ H_{\delta_1}^{k}f }_{L^2_{x,v}}+\norm{ H_{\delta_2}^{k}f }_{L^2_{x,v}}},
 	 	\end{align*}
 	where $C_3$ is a constant depending only on $T,\delta_1,\delta_2$.  Combining the above estimate   with \eqref{k++}, we conclude that
 	\begin{equation*}
 		\sup_{0<t\leq T}t^{(\lambda+1)k} \norm{\partial_{x_1}^{k}f(t)}_{L^2_{x,v}}\leq   \eps_0  (2C_3L)^{k} k!.
 	\end{equation*}
 Similarly, the above estimate is also true with $\partial_{x_1}$ replaced by $\partial_{x_2}$ or $\partial_{x_3}$. This, with the fact that
\begin{eqnarray*}
\forall\ \alpha\in\mathbb Z_+^3,\quad 	\norm{\partial_x^\alpha f}_{L^2_{x,v}}\leq \sum_{1\leq j\leq 3}\norm{\partial_{x_j}^{\abs\alpha}f}_{L^2_{x,v}},
\end{eqnarray*}
gives
\begin{eqnarray*}
	\forall\ \alpha\in\mathbb Z_+^3,\quad 	\sup_{0<t\leq T}t^{(\lambda+1)\abs\alpha}\norm{\partial_x^\alpha f(t)}_{L^2_{x,v}}\leq \eps_0  (6C_3L)^{\abs\alpha}\abs\alpha!.
\end{eqnarray*}
In the same way we have
  \begin{eqnarray*}
\forall\ \beta\in\mathbb Z_+^3, \quad 	\sup_{0<t\leq T}t^{ \lambda \abs\beta} \norm{\partial_{v}^{\beta}f(t)}_{L^2_{x,v}}    \leq \eps_0  (6C_3L)^{\abs\beta} \abs\beta!.
 	\end{eqnarray*}
 	Consequently,  for any $\alpha,\beta\in\mathbb Z_+^3,$
 	\begin{equation}\label{tequ}
 		\begin{aligned}
 		&	\sup_{0<t\leq T}t^{(\lambda +1)\abs\alpha+ \lambda \abs\beta} \norm{\partial_x^{\alpha}\partial_{v}^{\beta}f(t)}_{L^2_{x,v}}   \\
 		&\leq \sup_{0<t\leq T} \Big(t^{2(\lambda +1)\abs\alpha} \norm{\partial_{x}^{2\alpha}f(t)}_{L^2_{x,v}}\Big)^{1/2}\Big(t^{2 \lambda \abs\beta}\norm{\partial_{v}^{2\beta}f(t)}_{L^2_{x,v}}\Big)^{1/2} \\
 			&\leq \eps_0  (6C_3L)^{ \abs\alpha+\abs\beta }\Big(\abs{2\alpha}! \abs{2\beta}!\Big)^{1/2}\leq \eps_0  (12C_3L)^{ \abs\alpha+\abs\beta }\inner{ \abs\alpha+\abs\beta }!,
 		\end{aligned}
 	\end{equation}
the last inequality using  the fact that $p!q!\leq(p+q)!\leq 2^{p+q}p!q!$ for any $p, q \in \mathbb{Z}$. Thus    the desired estimate \eqref{anestimate+}  follows from \eqref{tequ} by choosing $C$ large enough such that $C>12C_3L+1$.  We have proven Theorem \ref{mainresult} for  the strong angular singularity condition that $1/2\leq s<1$.

\section{Optimal Gevrey smoothing effect for mild angular singularity}
\label{sec:mild}

This part focus on the mild angular singularity case, i.e., $0<s<1/2$ in \eqref{angu}. In this case, we can expect  Gevrey class regularity with optimal   Gevrey   index $\frac{1}{2s}$.

    \begin{theorem}\label{thm:Gevrey}
    	Assume that the cross-section satisfies \eqref{kern} and \eqref{angu} with $\gamma \geq 0 $ and $ 0<s <1/2$.   Let $T\geq 1$ be arbitrarily given, and
 	 let $f\in L_m^1L_T^\infty L_v^2$  be any solution to the Cauchy problem \eqref{eqforper} satisfying \eqref{higher+}. Moreover, let   $\lambda$ be an arbitrarily given number satisfying \eqref{relation} and let $H_{\delta_1}$ and $H_{\delta_2}$ be two  vector fields defined by \eqref{vecM}, with $\delta_j$  defined in terms of $\lambda$ by \eqref{de1de2}.
 	 Then  there exists a sufficiently small constant $\eps_0>0$ and a large constant $L\geq 1$, with $L$ depending only on $T,\lambda,$ and the numbers $c_0$ and $C_0$ in Section \ref{sec:prelim},  such that if
 	  \begin{equation*}
 \int_{{\mathbb Z}^3}  \Big( \sup_{0<t\leq T}   \norm{ {\hat f (t, m )}}_{L^2_v}\Big) d \Sigma(m)+\int_{{\mathbb Z}^3}   \Big( \int_0^T  \normm{ \hat f (t,m )} ^2dt \Big)^{\frac{1}{2}} d \Sigma(m) \leq \eps_0,
 \end{equation*}
 then the  estimate
 	\begin{multline}\label{+k+++}
 \sum_{1\leq j\leq 2}	\int_{\mathbb{Z}^3}\sup\limits_{0 < t \leq T}\norm{\widehat{H_{\delta_j}^{k}f}(t,m)}_{L^2_v}d\Sigma(m)\\
 	+\sum_{1\leq j\leq 2} \int_{\mathbb{Z}^3}\left(\int_{0}^{T}\normm{  \widehat{H_{\delta_j}^{k}f}(t,m)}^{2}dt\right)^{1/2}d\Sigma(m) \leq \frac{\eps_0 L^{k } (k!)^{\frac{1}{2s}}}{(k+1)^2}
 	\end{multline}
 	holds true  for  any $k\in\mathbb Z_+$.
\end{theorem}

\begin{proof}
	[Sketch of the proof of Theorem \ref{thm:Gevrey}] The proof is similar as that of Theorem \ref{theorem}. So for brevity we only sketch the proof, emphasizing the difference.  In the following argument, we always assume that $0<s< \frac 12$, and denote by $C$  different generic constants,   depending only on $T, \lambda,$ and the numbers $c_0,C_0$ in Section \ref{sec:prelim}.
	
	  As in the previous section we use induction on $k$ to prove \eqref{+k+++}. Suppose that for given $k\geq 1$, the estimate
	\begin{multline}\label{+suppose}
\sum_{1\leq j\leq 2}	\int_{\mathbb{Z}^3}\sup\limits_{0 < t \leq T}\norm{\widehat{H_{\delta_j}^{\ell}f}(t,m)}_{L^2_v}d\Sigma(m)\\
 	+\sum_{1\leq j\leq 2} \int_{\mathbb{Z}^3}\left(\int_{0}^{T}\normm{  \widehat{H_{\delta_j}^{\ell}f}(t,m)}^{2}dt\right)^{1/2}d\Sigma(m)  \leq \frac{\eps_0L^{\ell} (\ell!)^{\frac{1}{2s}}}{(\ell+1)^2}
\end{multline}
holds true for any $\ell\leq k-1$.  We will prove the above estimate is still valid for $\ell=k$.
Repeating the argument before \eqref{enestimate}, we have the following estimate similar to \eqref{enestimate}:
 \begin{equation}\label{hk1+}
		\begin{aligned}
&\frac{1}{2}\frac{d}{dt}\sum_{1\leq j\leq 2}\norm{\widehat{H_{\delta_j}^{k}f}}_{L^2_v}^2+c_0\sum_{1\leq j\leq 2}\normm{ \widehat{H_{\delta_j}^{k}f}}^{2}\\
&\leq \sum_{1\leq j\leq 2} \norm{\widehat{H_{\delta_j}^{k}f}}_{L^2_v}^{2}+\sum_{1\leq j\leq 2}\delta_j k  t^{\delta_j-1} \big|\big( \partial_{v_1} \widehat{H_{\delta_j}^{k-1}f},\  \widehat{H_{\delta_j}^{k}f}\big)_{L^2_v}\big|
\\&\quad+\sum_{1\leq j\leq 2} \big|\big(\mathcal{F}_x([ H_{\delta_j}^{k},\ \mathcal{L}]f), \widehat{H_{\delta_j}^{k}f}\big)_{L^2_v}\big|+\sum_{1\leq j\leq 2}\big|\big(\mathcal{F}_x(H_{\delta_j}^{k}\Gamma(f,f)), \ \widehat{H_{\delta_j}^{k}f}\big)_{L^2_v}\big|.
\end{aligned}
	\end{equation}
	It suffices to deal with the second term on the right-hand side, since the other terms can be treated in the same way as that in the previous case of $1/2\leq s<1.$
	
For each $j=1,2,$ and for any $\eps >0$, we have
\begin{equation}\label{interpo}
\begin{aligned}
&k  t^{\delta_j-1}\big|\big(\partial_{v_1} \widehat{H_{\delta_j}^{k-1}f},\ \widehat{H_{\delta_j}^{k}f}\big)_{L^2_v}\big|   \leq k  t^{\delta_j-1}\norm{\partial_{v_1} \widehat{H_{\delta_j}^{k-1}f}}_{H_v^{-s}}\norm{\widehat{H_{\delta_j}^{k}f}}_{H_v^{s}}\\
&\qquad\qquad \qquad\qquad\qquad \leq \eps   \normm{\widehat{H_{\delta_j}^{k}f}}^2+ C \eps^{-1}  k^2  t^{2(\delta_j-1)}\norm{\partial_{v_1} \widehat{H_{\delta_j} ^{k-1}f}}_{H_v^{-s}}^2,
\end{aligned}
\end{equation}
the last inequality using \eqref{+lowoftri}.
Moreover, recalling $0<2s<1$,  we use  the interpolation inequality that,
$$
 \forall\ \tilde\eps>0,\quad \|g\|^2_{H_v^{-s}}\leq \tilde\eps  \| g\|^{2}_{H^{s}_v}+\tilde\eps^{-\frac{1-2s}{2s}}\| g\|^{2}_{H^{s-1}_v},
$$
with $\tilde \eps =\eps^2 t^{2\delta_1}k^{-2}  t^{-2(\delta_j-1)}$ and $g=\partial_{v_1}\widehat{ H_{\delta_j}^{k-1}f}$; this gives
\begin{equation}
\begin{aligned}\label{eqpre}
	&\eps^{-1}k^2  t^{2(\delta_j-1)}\norm{\partial_{v_1}\widehat{ H_{\delta_j}^{k-1}f}}_{H_v^{-s}}^2\\
	&\leq  \eps  t^{2\delta_1}\norm{\partial_{v_1}\widehat{ H_{\delta_j}^{k-1}f}}_{H_v^{s}}^2+\eps^{\frac{s-1}{s}}k^{\frac{1}{s}}  t^{\frac{1}{s}(\delta_j-1)} t^{-\frac{1-2s }{s}\delta_1}\norm{\partial_{v_1}\widehat{ H_{\delta_j}^{k-1}f}}_{H_v^{s-1}}^2\\
	&\leq  \eps  \norm{t^{\delta_1}\partial_{v_1}\widehat{ H_{\delta_j}^{k-1}f}}_{H_v^{s}}^2+C\eps^{\frac{s-1}{s}}k^{\frac{1}{s}}  t^{\frac{1}{s}\big (\delta_j-1-(1-2s)\delta_1\big  )}  \normm{\widehat{ H_{\delta_j}^{k-1}f}}^2,
\end{aligned}
\end{equation}
the last inequality using again \eqref{+lowoftri}. As for the last term on the right-hand side of \eqref{eqpre}, we use  the definition  \eqref{de1de2}  of $\delta_j$   and  the fact that $\delta_1>\delta_2$ in view of \eqref{lodelta}, to compute, for $j=1, 2,$
 \begin{align*}
	\delta_j-1-(1-2s)\delta_1 \geq \delta_2-1-(1-2s)\delta_1\geq  2s +(1-2s)\lambda-(1-2s)\lambda \geq 0,
\end{align*}
which yields
\begin{eqnarray*}
\forall\ 0<t\leq T,\quad \eps^{\frac{s-1}{s}}k^{\frac{1}{s}}  t^{\frac{1}{s}\big (\delta_j-1-(1-2s)\delta_1\big  )}  \normm{\widehat{ H_{\delta_j}^{k-1}f}}^2\leq C	 \eps^{\frac{s-1}{s}}k^{\frac{1}{s}}    \normm{\widehat{ H_{\delta_j}^{k-1}f}}^2,
\end{eqnarray*}
and thus, substituting the above inequality into \eqref{eqpre},
\begin{equation*}
	\eps^{-1}k^2  t^{2(\delta_j-1)}\norm{\partial_{v_1}\widehat{ H_{\delta_j}^{k-1}f}}_{H_v^{-s}}^2\leq  \eps  \norm{t^{\delta_1}\partial_{v_1}\widehat{ H_{\delta_j}^{k-1}f}}_{H_v^{s}}^2+ C	 \eps^{\frac{s-1}{s}}k^{\frac{1}{s}}    \normm{\widehat{ H_{\delta_j}^{k-1}f}}^2.
\end{equation*}
Consequently, we combine the above
  estimate with \eqref{interpo} to obtain that, for any $\eps>0$ and   any $t\in ]0,T],$
\begin{multline}\label{coet}
k  t^{\delta_j-1}\big|\big(\partial_{v_1} \widehat{H_{\delta_j}^{k-1}f},\ \widehat{H_{\delta_j}^{k}f}\big)_{L^2_v}\big| \\
  \leq \eps  \normm{\widehat{H_{\delta_j}^{k}f}}^2+  \eps  \norm{t^{\delta_1}\partial_{v_1}\widehat{ H_{\delta_j}^{k-1}f}}_{H_v^{s}}^2+C\eps^{\frac{s-1}{s}}k^{\frac{1}{s}}    \normm{\widehat{ H_{\delta_j}^{k-1}f}}^2.
\end{multline}
As for the second term on the right-hand side of \eqref{coet}, we first use the second equation in \eqref{generate} and then the fact that
\begin{align*}
\forall\ (m,\eta) \in\mathbb Z^3\times\mathbb R^3,\quad 	|a(m,\eta)b(m,\eta)^{k-1}|^2\leq | a(m,\eta)|^{2k }+| b(m,\eta) |^{2k},
\end{align*}
   to compute
\begin{align*}	\norm{t^{\delta_1}\partial_{v_1}\widehat{ H_{\delta_j}^{k-1}f}}_{H_v^{s}}^2&=	\norm{\mathcal F_x \big(t^{\delta_1}\partial_{v_1} H_{\delta_j}^{k-1}f\big)}_{H_v^{s}}^2\\
	&\leq C\norm{ \mathcal F_x\big(H_{\delta_1} H_{\delta_j}^{k-1}f\big)}_{H_v^{s}}^2+C\norm{ \mathcal F_x\big(H_{\delta_2} H_{\delta_j}^{k-1}f\big)}_{H_v^{s}}^2\\
	&\leq C\norm{\widehat{ H_{\delta_1}^{k}f}}_{H_v^{s}}^2+C\norm{\widehat{ H_{\delta_2}^{k}f}}_{H_v^{s}}^2\leq C\sum_{1\leq j\leq 2}\normm{\widehat{ H_{\delta_j}^{k}f}}^2,
\end{align*}
the last inequality following from  \eqref{+lowoftri}.  Substituting the above estimate into    \eqref{coet} we conclude that, for any $\eps>0$ and for each $j=1,2,$
\begin{equation*}
	k  t^{\delta_j-1}\big|\big(\partial_{v_1} \widehat{H_{\delta_j}^{k-1}f},\ \widehat{H_{\delta_j}^{k}f}\big)_{L^2_v}\big| \\
	 \leq C \eps \sum_{1\leq j\leq 2}\normm{\widehat{ H_{\delta_j}^{k}f}}^2 +C\eps^{\frac{s-1}{s}}k^{\frac{1}{s}}    \normm{\widehat{ H_{\delta_j}^{k-1}f}}^2,
\end{equation*}
which  with \eqref{hk1+}  yields that, for any $\eps>0,$
\begin{equation*}
		\begin{aligned}
&\frac{1}{2}\frac{d}{dt}\sum_{1\leq j\leq 2}\norm{\widehat{H_{\delta_j}^{k}f}}_{L^2_v}^2+c_0\sum_{1\leq j\leq 2}\normm{ \widehat{H_{\delta_j}^{k}f}}^{2}\\
&\leq \sum_{1\leq j\leq 2} \norm{\widehat{H_{\delta_j}^{k}f}}_{L^2_v}^{2}+C \eps \sum_{1\leq j\leq 2}\normm{\widehat{ H_{\delta_j}^{k}f}}^2 +C\eps^{\frac{s-1}{s}}k^{\frac{1}{s}}  \sum_{1\leq j\leq 2}   \normm{\widehat{ H_{\delta_j}^{k-1}f}}^2
\\&\quad+\sum_{1\leq j\leq 2} \big|\big(\mathcal{F}_x([ H_{\delta_j}^{k},\ \mathcal{L}]f), \widehat{H_{\delta_j}^{k}f}\big)_{L^2_v}\big|+\sum_{1\leq j\leq 2}\big|\big(\mathcal{F}_x(H_{\delta_j}^{k}\Gamma(f,f)), \ \widehat{H_{\delta_j}^{k}f}\big)_{L^2_v}\big|.
\end{aligned}
\end{equation*}
Letting $\eps$   above   be  sufficiently small,  we get that
\begin{equation*}
		\begin{aligned}
&\frac{1}{2}\frac{d}{dt}\sum_{1\leq j\leq 2}\norm{\widehat{H_{\delta_j}^{k}f}}_{L^2_v}^2+\frac{c_0}{2}\sum_{1\leq j\leq 2}\normm{ \widehat{H_{\delta_j}^{k}f}}^{2} \leq \sum_{1\leq j\leq 2} \norm{\widehat{H_{\delta_j}^{k}f}}_{L^2_v}^{2}+C k^{\frac{1}{s}}    \sum_{1\leq j\leq 2} \normm{\widehat{ H_{\delta_j}^{k-1}f}}^2
\\&\qquad\qquad+\sum_{1\leq j\leq 2} \big|\big(\mathcal{F}_x([ H_{\delta_j}^{k},\ \mathcal{L}]f), \widehat{H_{\delta_j}^{k}f}\big)_{L^2_v}\big|+\sum_{1\leq j\leq 2}\big|\big(\mathcal{F}_x(H_{\delta_j}^{k}\Gamma(f,f)), \ \widehat{H_{\delta_j}^{k}f}\big)_{L^2_v}\big|.
\end{aligned}
\end{equation*}
Note that the above estimate is quite similar to  \eqref{hkenergy+}, with the factor $k^2$ therein replaced by $k^{1/s}$ here.   Moreover, observe that
  \begin{multline*}
	 k^{\frac{1}{2s}}\sum_{1\leq j\leq 2}  \int_{\mathbb Z^3} \Big( \int_{0}^{T}\normm{\widehat{H_{\delta_j}^{k-1}f}(t,m)}^2dt\Big)^{1\over2} d\Sigma(m)\\
	\leq  k^{\frac{1}{2s}}\frac{\eps_0L^{k-1}[(k-1)!]^{\frac{1}{2s}}}{k^2}\leq C\frac{\eps_0L^{k-1}(k!)^{\frac{1}{2s}}}{(k+1)^2},
\end{multline*}
which just follows from inductive assumption \eqref{+suppose}. Thus we may repeat the argument after \eqref{hkenergy+} and use the above estimate  instead of \eqref{srt}, to conclude that
\begin{align*}
	&\sum_{1\leq j\leq 2}\bigg[\int_{\mathbb Z^3} \sup\limits_{0<t\leq T}\norm{\widehat{H_{\delta_j}^{k}f}(t,m)}_{L^2_v}d\Sigma(m)+   \int_{\mathbb Z^3}\Big( \int_{0}^{T}\normm{ \widehat{H_{\delta_j}^{k}f}(t,m)}^{2}dt\Big)^{1/2}d\Sigma(m)\bigg]\\
	&\leq \frac{\eps_0L^{k}( k!)^{\frac{1}{2s}}}{(k+1)^2}.
\end{align*}
Then \eqref{+suppose} holds for $\ell=k$, and thus \eqref{+k+++} follows. The proof of
  Theorem \ref{thm:Gevrey} is completed.
\end{proof}

   \begin{proof}
   	[Completing the proof of Theorem \ref{mainresult}: Gevrey   smoothing effect for $  0< s<\frac12$] \   With  the help of \eqref{+k+++},  the Gevrey estimate \eqref{alpha1} for  $0<s<\frac12$   just follows from the same argument as that in Subsection \ref{subsec:ana}. So  we omit it for brevity.
   \end{proof}

\bigskip
\noindent {\bf Acknowledgements.} W.-X. Li   was supported by NSF of China (Nos.11961160716, 12131017, 12221001)  and  the  Natural Science Foundation of Hubei Province (No.
2019CFA007). C.-J. Xu was supported by the NSFC (No.12031006) and the Fundamental
Research Funds for the Central Universities of China.


\begin{thebibliography}{10}

\bibitem{MR1765272}
R.~Alexandre, L.~Desvillettes, C.~Villani, and B.~Wennberg.
\newblock Entropy dissipation and long-range interactions.
\newblock {\em Arch. Ration. Mech. Anal.}, 152(4):327--355, 2000.

\bibitem{MR3950012}
R.~Alexandre, F.~H\'{e}rau, and W.-X. Li.
\newblock Global hypoelliptic and symbolic estimates for the linearized
  {B}oltzmann operator without angular cutoff.
\newblock {\em J. Math. Pures Appl. (9)}, 126:1--71, 2019.

\bibitem{MR2679369}
R.~Alexandre, Y.~Morimoto, S.~Ukai, C.-J. Xu, and T.~Yang.
\newblock Regularizing effect and local existence for the non-cutoff
  {B}oltzmann equation.
\newblock {\em Arch. Ration. Mech. Anal.}, 198(1):39--123, 2010.

\bibitem{MR2847536}
R.~Alexandre, Y.~Morimoto, S.~Ukai, C.-J. Xu, and T.~Yang.
\newblock The {B}oltzmann equation without angular cutoff in the whole space:
  qualitative properties of solutions.
\newblock {\em Arch. Ration. Mech. Anal.}, 202(2):599--661, 2011.

\bibitem{MR2863853}
R.~Alexandre, Y.~Morimoto, S.~Ukai, C.-J. Xu, and T.~Yang.
\newblock The {B}oltzmann equation without angular cutoff in the whole space:
  {I}, {G}lobal existence for soft potential.
\newblock {\em J. Funct. Anal.}, 262(3):915--1010, 2012.

\bibitem{MR2959943}
R.~Alexandre, Y.~Morimoto, S.~Ukai, C.-J. Xu, and T.~Yang.
\newblock Smoothing effect of weak solutions for the spatially homogeneous
  {B}oltzmann equation without angular cutoff.
\newblock {\em Kyoto J. Math.}, 52(3):433--463, 2012.

\bibitem{MR4201411}
R.~Alonso, Y.~Morimoto, W.~Sun, and T.~Yang.
\newblock Non-cutoff {B}oltzmann equation with polynomial decay perturbations.
\newblock {\em Rev. Mat. Iberoam.}, 37(1):189--292, 2021.

\bibitem{MR4526062}
R.~Alonso, Y.~Morimoto, W.~Sun, and T.~Yang.
\newblock De {G}iorgi argument for weighted {$L^2\cap L^\infty$} solutions to
  the non-cutoff {B}oltzmann equation.
\newblock {\em J. Stat. Phys.}, 190(2):Paper No. 38, 98, 2023.

\bibitem{MR3665667}
J.-M. Barbaroux, D.~Hundertmark, T.~Ried, and S.~Vugalter.
\newblock Gevrey smoothing for weak solutions of the fully nonlinear
  homogeneous {B}oltzmann and {K}ac equations without cutoff for {M}axwellian
  molecules.
\newblock {\em Arch. Ration. Mech. Anal.}, 225(2):601--661, 2017.

\bibitem{MR1949176}
F.~Bouchut.
\newblock Hypoelliptic regularity in kinetic equations.
\newblock {\em J. Math. Pures Appl. (9)}, 81(11):1135--1159, 2002.

\bibitem{MR4612704}
H.~Cao, W.-X. Li, and C.-J. Xu.
\newblock Analytic smoothing effect of the spatially inhomogeneous {L}andau
  equations for hard potentials.
\newblock {\em J. Math. Pures Appl. (9)}, 176:138--182, 2023.

\bibitem{MR4375857}
H.~Chen, X.~Hu, W.-X. Li, and J.~Zhan.
\newblock The {G}evrey smoothing effect for the spatially inhomogeneous
  {B}oltzmann equations without cut-off.
\newblock {\em Sci. China Math.}, 65(3):443--470, 2022.

\bibitem{MR2514370}
H.~Chen, W.~Li, and C.~Xu.
\newblock Gevrey regularity for solution of the spatially homogeneous {L}andau
  equation.
\newblock {\em Acta Math. Sci. Ser. B (Engl. Ed.)}, 29(3):673--686, 2009.

\bibitem{MR2467026}
H.~Chen, W.-X. Li, and C.-J. Xu.
\newblock Gevrey hypoellipticity for linear and non-linear {F}okker-{P}lanck
  equations.
\newblock {\em J. Differential Equations}, 246(1):320--339, 2009.

\bibitem{MR2557895}
H.~Chen, W.-X. Li, and C.-J. Xu.
\newblock Analytic smoothness effect of solutions for spatially homogeneous
  {L}andau equation.
\newblock {\em J. Differential Equations}, 248(1):77--94, 2010.

\bibitem{MR2763329}
H.~Chen, W.-X. Li, and C.-J. Xu.
\newblock Gevrey hypoellipticity for a class of kinetic equations.
\newblock {\em Comm. Partial Differential Equations}, 36(4):693--728, 2011.

\bibitem{MR2506070}
Y.~Chen, L.~Desvillettes, and L.~He.
\newblock Smoothing effects for classical solutions of the full {L}andau
  equation.
\newblock {\em Arch. Ration. Mech. Anal.}, 193(1):21--55, 2009.

\bibitem{MR1324404}
L.~Desvillettes.
\newblock About the regularizing properties of the non-cut-off {K}ac equation.
\newblock {\em Comm. Math. Phys.}, 168(2):417--440, 1995.

\bibitem{MR1737547}
L.~Desvillettes and C.~Villani.
\newblock On the spatially homogeneous {L}andau equation for hard potentials.
  {I}. {E}xistence, uniqueness and smoothness.
\newblock {\em Comm. Partial Differential Equations}, 25(1-2):179--259, 2000.

\bibitem{MR2038147}
L.~Desvillettes and B.~Wennberg.
\newblock Smoothness of the solution of the spatially homogeneous {B}oltzmann
  equation without cutoff.
\newblock {\em Comm. Partial Differential Equations}, 29(1-2):133--155, 2004.

\bibitem{MR4356815}
R.~Duan, W.-X. Li, and L.~Liu.
\newblock Gevrey regularity of mild solutions to the non-cutoff {B}oltzmann
  equation.
\newblock {\em Adv. Math.}, 395:Paper No. 108159, 2022.

\bibitem{MR4230064}
R.~Duan, S.~Liu, S.~Sakamoto, and R.~M. Strain.
\newblock Global mild solutions of the {L}andau and non-cutoff {B}oltzmann
  equations.
\newblock {\em Comm. Pure Appl. Math.}, 74(5):932--1020, 2021.

\bibitem{MR1026858}
C.~Foias and R.~Temam.
\newblock Gevrey class regularity for the solutions of the {N}avier-{S}tokes
  equations.
\newblock {\em J. Funct. Anal.}, 87(2):359--369, 1989.

\bibitem{MR3485915}
L.~Glangetas, H.-G. Li, and C.-J. Xu.
\newblock Sharp regularity properties for the non-cutoff spatially homogeneous
  {B}oltzmann equation.
\newblock {\em Kinet. Relat. Models}, 9(2):299--371, 2016.

\bibitem{MR2784329}
P.~T. Gressman and R.~M. Strain.
\newblock Global classical solutions of the {B}oltzmann equation without
  angular cut-off.
\newblock {\em J. Amer. Math. Soc.}, 24(3):771--847, 2011.

\bibitem{MR2807092}
P.~T. Gressman and R.~M. Strain.
\newblock Sharp anisotropic estimates for the {B}oltzmann collision operator
  and its entropy production.
\newblock {\em Adv. Math.}, 227(6):2349--2384, 2011.

\bibitem{MR3102561}
F.~H\'{e}rau and W.-X. Li.
\newblock Global hypoelliptic estimates for {L}andau-type operators with
  external potential.
\newblock {\em Kyoto J. Math.}, 53(3):533--565, 2013.

\bibitem{MR4107942}
F.~H\'{e}rau, D.~Tonon, and I.~Tristani.
\newblock Regularization estimates and {C}auchy theory for inhomogeneous
  {B}oltzmann equation for hard potentials without cut-off.
\newblock {\em Comm. Math. Phys.}, 377(1):697--771, 2020.

\bibitem{MR4033752}
C.~Imbert, C.~Mouhot, and L.~Silvestre.
\newblock Decay estimates for large velocities in the {B}oltzmann equation
  without cutoff.
\newblock {\em J. \'{E}c. polytech. Math.}, 7:143--184, 2020.

\bibitem{MR4049224}
C.~Imbert and L.~Silvestre.
\newblock The weak {H}arnack inequality for the {B}oltzmann equation without
  cut-off.
\newblock {\em J. Eur. Math. Soc. (JEMS)}, 22(2):507--592, 2020.

\bibitem{MR4229202}
C.~Imbert and L.~Silvestre.
\newblock The {S}chauder estimate for kinetic integral equations.
\newblock {\em Anal. PDE}, 14(1):171--204, 2021.

\bibitem{MR4433077}
C.~Imbert and L.~E. Silvestre.
\newblock Global regularity estimates for the {B}oltzmann equation without
  cut-off.
\newblock {\em J. Amer. Math. Soc.}, 35(3):625--703, 2022.

\bibitem{MR3348825}
N.~Lerner, Y.~Morimoto, K.~Pravda-Starov, and C.-J. Xu.
\newblock Gelfand-{S}hilov and {G}evrey smoothing effect for the spatially
  inhomogeneous non-cutoff {K}ac equation.
\newblock {\em J. Funct. Anal.}, 269(2):459--535, 2015.

\bibitem{MR3193940}
W.-X. Li.
\newblock Global hypoelliptic estimates for fractional order kinetic equation.
\newblock {\em Math. Nachr.}, 287(5-6):610--637, 2014.

\bibitem{MR4147430}
W.-X. Li and L.~Liu.
\newblock Gelfand-{S}hilov smoothing effect for the spatially inhomogeneous
  {B}oltzmann equations without cut-off.
\newblock {\em Kinet. Relat. Models}, 13(5):1029--1046, 2020.

\bibitem{MR3456819}
W.-X. Li, P.~Luo, and S.~Tian.
\newblock {$L^2$}-regularity of kinetic equations with external potential.
\newblock {\em J. Differential Equations}, 260(7):5894--5911, 2016.

\bibitem{MR1278244}
P.-L. Lions.
\newblock On {B}oltzmann and {L}andau equations.
\newblock {\em Philos. Trans. Roy. Soc. London Ser. A}, 346(1679):191--204,
  1994.

\bibitem{MR3177625}
Y.~Morimoto, K.~Pravda-Starov, and C.-J. Xu.
\newblock A remark on the ultra-analytic smoothing properties of the spatially
  homogeneous {L}andau equation.
\newblock {\em Kinet. Relat. Models}, 6(4):715--727, 2013.

\bibitem{MR3310275}
Y.~Morimoto, S.~Wang, and T.~Yang.
\newblock A new characterization and global regularity of infinite energy
  solutions to the homogeneous {B}oltzmann equation.
\newblock {\em J. Math. Pures Appl. (9)}, 103(3):809--829, 2015.

\bibitem{MR3572500}
Y.~Morimoto, S.~Wang, and T.~Yang.
\newblock Measure valued solutions to the spatially homogeneous {B}oltzmann
  equation without angular cutoff.
\newblock {\em J. Stat. Phys.}, 165(5):866--906, 2016.

\bibitem{MR2523694}
Y.~Morimoto and C.-J. Xu.
\newblock Ultra-analytic effect of {C}auchy problem for a class of kinetic
  equations.
\newblock {\em J. Differential Equations}, 247(2):596--617, 2009.

\bibitem{MR3325244}
Y.~Morimoto and T.~Yang.
\newblock Smoothing effect of the homogeneous {B}oltzmann equation with measure
  valued initial datum.
\newblock {\em Ann. Inst. H. Poincar\'{e} C Anal. Non Lin\'{e}aire},
  32(2):429--442, 2015.

\bibitem{MR3551261}
L.~Silvestre.
\newblock A new regularization mechanism for the {B}oltzmann equation without
  cut-off.
\newblock {\em Comm. Math. Phys.}, 348(1):69--100, 2016.

\bibitem{MR4431674}
L.~Silvestre and S.~Snelson.
\newblock Solutions to the non-cutoff {B}oltzmann equation uniformly near a
  {M}axwellian.
\newblock {\em Math. Eng.}, 5(2):Paper No. 034, 36, 2023.

\bibitem{MR839310}
S.~Ukai.
\newblock Local solutions in {G}evrey classes to the nonlinear {B}oltzmann
  equation without cutoff.
\newblock {\em Japan J. Appl. Math.}, 1(1):141--156, 1984.

\end{thebibliography}
\end{document}